\documentclass[10pt]{amsart}
\usepackage{amsmath,amsfonts,amssymb,amsthm,subfigure,verbatim,
graphicx,tikz}
\usepackage{amsrefs}
\usepackage{mathrsfs}

\theoremstyle{plain}
\begingroup 
\newtheorem{thm}{Theorem}[section]  
\newtheorem*{thm*}{Theorem}     
\newtheorem*{cor*}{Corollary}    
\newtheorem{cor}[thm]{Corollary}   
\newtheorem{lem}[thm]{Lemma}     
\newtheorem{prop}[thm]{Proposition} 

\theoremstyle{definition}
\newtheorem{ex}[thm]{Example}
\newtheorem{rem}[thm]{Remark}
\newtheorem{question}[thm]{Question}
\endgroup

\newcommand{\Z}{\mathbb{Z}}
\newcommand{\N}{\mathbb{N}}
\newcommand{\C}{\mathbb{C}}

\newcommand{\cb}{\mathscr B}

\title{Dynamical properties of some adic systems with arbitrary orderings}
\author{Sarah Frick}
\address{Department of Mathematics, Furman University, Greenville, SC 29613 USA}
\email{sarah.frick@furman.edu}
\author{Karl Petersen}
\address{Department of Mathematics,
CB 3250 Phillips Hall,
University of North Carolina,
Chapel Hill, NC 27599 USA}
\email{petersen@math.unc.edu}
\author{Sandi Shields}
\address{College of Charleston, 66 George St., Charleston, SC 29424-0001 USA}
\email{shieldss@cofc.edu}
\date{\today}

\begin{document}

\maketitle

\begin{abstract}
We consider arbitrary orderings of the edges entering each vertex of the (downward directed) Pascal graph. Each ordering determines an adic (Bratteli-Vershik) system, with a transformation that is defined on most of the space of infinite paths that begin at the root. We prove that for every ordering the coding of orbits according to the partition of the path space determined by the first three edges is essentially faithful, meaning that it is one-to-one on a set of paths that has full measure for every fully supported invariant probability measure. We also show that for every $k$ the subshift that arises from coding orbits according to the first $k$ edges is topologically weakly mixing. We give a necessary and sufficient condition for any adic system to be topologically conjugate to an odometer and use this condition to determine the probability that a random order on a fixed diagram, or a diagram constructed at random in some way, is topologically conjugate to an odometer. We also show that the closure of the union over all orderings of the subshifts arising from codings of the Pascal adic by the first edge has superpolynomial complexity, is not topologically transitive, and has no periodic points besides the two fixed points, {while the intersection over all orderings consists of just four orbits}.
\end{abstract}

\section{Introduction}
Adic, or Bratteli-Vershik (BV), systems represent all measure-preserving systems on nonatomic probability spaces and all minimal systems on the Cantor set. They provide a convenient combinatorial description of the construction of systems by the cutting and stacking method, present tail fields of stochastic processes as the fields of adic-invariant sets, facilitate the study of measure-preserving and topological orbit equivalence, and present important and interesting particular systems for detailed study. We refer to the introductions of the papers in our references list and the survey \cite{Durand2010} for background and further details.

For adic systems, invariant measures and orbit equivalence depend only on the graph-theoretic structure of the underlying diagram and are independent of the orders assigned to the edges entering each vertex, which determine the adic transformation; further dynamical properties, such as the spectrum, isomorphism, or joinings, do depend on the particular ordering.
Some recent work has investigated how the orderings assigned to the edges can affect the sets of maximal and minimal paths. Bezuglyi, Kwiatkowski, and Yassawi \cite{BezuglyiKwiatkowskiYassawi2014} studied arbitrary reorderings of adic systems to determine which ones---the ``perfect" orderings---allow the adic transformation to be defined as a homeomorphism. They showed that when orderings of a diagram of finite rank $r$ are chosen uniformly at random, there is a $J \in \{1,\dots ,r\}$ such that almost every ordering has $J$ maximal and $J$ minimal paths. Medynets \cite{Medynets2006} gave an example of a diagram for which no ordering is perfect. Bezuglyi and Yassawi \cite{BezuglyiYassawi2013} pursue this question beyond the finite rank case and give necessary and sufficient conditions for a diagram to admit a perfect ordering. Janssen, Quas, and Yassawi \cite{JanssenYassawiQuas2014} produce a class of systems for which $J =\infty$ and almost every ordering is imperfect.

It is important to determine when coding the orbits of paths according to the natural partition into cylinder sets determined by the first $k$ edges produces an essentially faithful representation of the system as a subshift on a finite alphabet; for then traditional methods of symbolic dynamics can be applied, and, even if the adic transformation cannot be defined as a homeomorphism, nevertheless the corresponding subshift can be studied as a legitimate dynamical system. By ``essential" we mean that the coding map is one-to-one except on a set of paths that has measure zero for each fully supported adic-invariant probability measure on the system. This property was proved for the Pascal adic with a particular edge ordering (namely one that is lexicographical from left to right at each vertex) in \cite{MelaPetersen2005} and for a wider class of systems in \cite{Frick2009}. On the other hand, Downarowicz and Maass \cite{DownarowiczMaass2008} (see also \cite{Hoynes2014}) showed that each minimal finite topological rank adic system, namely one that can be represented up to topological conjugacy by an adic system with a bounded number of vertices on each level, with unique maximal and minimal paths, and having all orbits dense, either has a faithful coding by the first $k$ edges for some $k$ or else is topologically conjugate to an odometer. This result may be regarded as a generalization of previous work (see \cites{Livshitz1988, Forrest1997, DurandHostSkau1999, Host2000}) on the relations among stationary adic systems, substitutions, and odometers. It was extended to aperiodic Cantor minimal systems in \cite{BezuglyiKwiatkowskiMedynets2009}.

 So far we have been lacking useful tests for when a general adic system allows an essentially faithful representation as a subshift or when it is topologically conjugate to an odometer. Here we show (Theorem \ref{thm:coding}) that for the Pascal system every edge ordering allows essentially faithful coding by the first three edges.
 Further, extending the ideas in \cite{Mela2002, MelaPetersen2005}, we show that for every ordering the subshift defined by the coding is topologically weakly mixing (Theorem \ref{thm:weakmixing}). We find this somewhat surprising, since the existence of eigenfunctions is highly sensitive to timing.
 We also establish a necessary and sufficient condition for an adic system of a fairly general kind to be topologically conjugate to an odometer (Theorem \ref{UOimpliesOdometer}). (Previously, sufficient conditions were known for systems with a bounded number of vertices per level \cite{DurandHostSkau1999, BezuglyiKwiatkowskiYassawi2014}. A similar condition has been given independently by Durand and Yassawi \cite{DurandYassawi2014}, and a necessary condition was mentioned in \cite{GjerdeJohansen2000}.) We use this condition to determine the probability that a random order on a fixed diagram, or an ordered diagram constructed at random in a certain sense, is topologically conjugate to an odometer (Examples \ref{ex:odometerprob} and \ref{ex:odometer2}).
We also study the subshift that is the closure of the union of the subshifts determined by the Pascal adic with all possible orderings: we estimate its complexity, prove that it is not topologically transitive, and show that its only periodic points are the two fixed points.
{On the other hand, the intersection over all orderings consists of just four orbits.}
 The final section presents examples and further observations related to the questions under discussion.

\section{The Pascal system}

In this section we fix notation and terminology for the Pascal adic system. Similar concepts apply also to other adic systems.

Let $P$ denote the Pascal graph, which is a directed graph with countably many vertices arranged into levels. The set of vertices is $\mathscr V = \{(x,y): x, y \in \mathbb Z_{\geq 0}\}$. The set $\mathscr E$ of directed edges consists of edges from $(x,y)$ to $(x+1,y)$ and $(x,y+1)$ for each $(x,y) \in \mathscr V$. The {\em level} of a vertex $(x,y)$ is $n=x+y$. For each $n \geq 0$ let
\begin{equation}
\mathscr V_n = \{(x,y): x + y = n\}
\end{equation}
denote the set of vertices at level $n$, and, for $n \geq 1$, $\mathscr E_n$ denote the set of edges with source in $\mathscr V_{n-1}$ and range in $\mathscr V_n$.

We specify an {\em ordering} of the incoming edges to all vertices other than $(0,0)$ by a function $\xi: \mathscr V \to \{0,1\}$: the edge from $(x-1,y)$ to $(x,y)$ is defined to be less than the edge from $(x,y-1)$ to $(x,y)$ if and only if $\xi (x,y)=1$. As usual with adic systems, each such ordering determines a partial order on the set $X$ of infinite directed paths on $P$, viewed as sequences of adjacent edges, that begin at the {\em root}, $(0,0)$. Paths $(\gamma_k)$ and $(\eta_k)$ are {\em comparable} if they eventually coincide: there is an $N$ such that $\gamma_k=\eta_k$ for all $k>N$.
If $n$ is the smallest such $N$, then $\gamma<\eta$ if and only if $\gamma_n < \eta_n$. Then the adic transformation $T_\xi$ is defined on the complement of the set $X_{\xi,\max}$ of paths that consist entirely of maximal edges by letting $T_\xi \gamma$ equal the smallest path $\eta$ that is larger than $\gamma$.
Similarly, $T_\xi^{-1}$ is defined on the complement of the set $X_{\xi,\min}$ of paths that consist entirely of minimal edges. Edges connecting ``boundary vertices'', i.e. from $(x,0)$ to $(x+1,0)$ and from $(0,y)$ to $(0,y+1)$, are both maximal and minimal. Define $\gamma_{-\infty}$ to be the path down the left boundary of the diagram and $\gamma_\infty$ to be the path down the right boundary.
The ordinary Pascal adic system is determined by the ordering $\xi \equiv 0$. When we are working with a fixed ordering, the subscript $\xi$ can be omitted. Denote by $X'_\xi$ the orbit under $T_\xi$ of $X_{\xi,\min} \cup X_{\xi,\max}$.
Sometimes $T_\xi$ can be extended to a map from (part of) $X_{\xi,\max}$ to $X_{\xi,\min}$.

For each $k \geq 1$ we consider the partition $\mathscr P_k$ of $X$
into the cylinder sets
\begin{equation}
 \text{Cyl}_\xi(a_1,\dots,a_k) = \{\gamma \in X: \gamma_i=a_i, i=1, \dots, k\}
\end{equation}
defined by specifying the first $k$ edges $a_i$, $i=1,\dots,k$.

These are clopen sets that generate the topology of $X$. For each $k$ there are $2^k$ cylinder sets of level $k$.
The elements of $\mathscr P_k$ may be regarded as symbols
and identified with the sets $ \text{Cyl}_\xi(a_1,\dots,a_k)$, the strings $a_1 \dots a_k$, or the corresponding initial paths of length $k$.

The set $\mathscr P_k^{\mathbb Z}$ of two-sided infinite sequences with entries from $\mathscr P_k$ is a compact metric space with distance $d(\omega, \omega') = 1/(2^{\min\{|m|:\omega_m \neq (\omega')_m\}}$).

For a fixed ordering $\xi$, each path $\gamma \in X$ has a {$k$-coding} $\omega_k(\gamma) \in \mathscr P_k^{\mathbb Z}$ given by $\omega_k(\gamma)_n = C$ if and only if $T_\xi^n \gamma \in C \in \mathscr P_k$.
We write $\omega \in \mathscr P_k^\Z$ as $\omega= \dots C_{-2}C_{-1}.C_0C_1 \dots .$
The subshift $(\Sigma_k, \sigma)$ consists of the set $\Sigma_k \subset \mathscr P_k^{\mathbb Z}$ of all sequences each of whose finite subblocks can be found in $\omega_k(\gamma)$ for some $\gamma \in X$, together with the shift transformation $\sigma$ defined by $(\sigma \omega)_n = \omega_{n+1}$. There are natural factor maps from $(\Sigma_k,\sigma)$ to $(\Sigma_l,\sigma)$ for each $l<k$.

 We say the $k$-coding is {\em faithful} if $\omega_k(\gamma)\neq{\omega_k(\gamma')}$ whenever $\gamma\neq{\gamma'}$. The system is expansive if and only if there is a $k$ such that the $k$-coding is faithful.

The fully supported nonatomic adic-invariant ergodic probability measures on the path space $X$ are a one-parameter family $\{ \mu_\alpha : 0 < \alpha < 1\}$. Each edge from $(x,y-1)$ to $(x,y)$ is assigned weight $\alpha$, and each edge from $(x-1,y)$ to $(x,y)$ is assigned weight $1-\alpha$. The measure of a cylinder set $ \text{Cyl}_\xi(a_1, \dots, a_k)$ as above is then the product of the weights on the edges that define the cylinder set. For a fixed vertex $v$, all cylinder sets determined by paths from $(0,0)$ to $v$ then have the same measure (see \cite{MelaPetersen2005}). Changing the ordering of the edges just changes the order in which paths, or cylinder sets, are mapped to one another and does not change the set of fully supported adic-invariant ergodic probability measures.

\begin{prop}\label{prop:maxpaths}
For each ordering $\xi$ of the Pascal graph and each fully supported adic-invariant probability measure $\mu_\alpha$, the set $X'_\xi=\mathscr O(X_{\xi,\min} \cup X_{\xi,\max})$ has measure 0.
\end{prop}
\begin{proof}
Let $s = \max\{\alpha,1-\alpha\}$.
For each vertex $v \in \mathscr V$ there is a unique (finite) reverse path from $v$ to the root $(0,0)$ that follows only minimal edges. (The corresponding cylinder set $C_{\min}(v)$ corresponds to the base of a column in the cutting and stacking representation of the system.) If $v$ is at level $n$, then $\mu_\alpha(C_{\min}(v)) \leq s^n$. Thus
\begin{equation}
\mu_\alpha \left( \bigcup_{v \in \mathscr V_n} C_{\min}(v) \right) \leq (n+1) s^n .
\end{equation}
Since each infinite path that follows only minimal edges is in the decreasing intersection of the sets $( \cup_{v \in \mathscr V_n} C_{\min}(v) )$, the set of minimal paths has $\mu_\alpha$ measure 0.
Similarly, the set of maximal paths has measure $0$, so $X'_\xi$ is the countable union of sets of measure $0$.
\end{proof}

We assign {\em basic blocks} $B_\xi(x,y)$ on the alphabet $A=\{a,b\}$ to the vertices $(x,y)$ of $P$ as follows. The vertices $(i,0)$, for all $i \geq 1$, are assigned the block $a$; the vertices $(0,j)$, for all $j \geq 1$, are assigned the block $b$; and each vertex $(x,y)$ for which both $x$ and $y$ are at least 1 is assigned the block $B_\xi(x,y-1)B_\xi(x-1,y)$ if $\xi(x,y)=0$, and $B_\xi(x-1,y)B_\xi(x,y-1)$ if $\xi(x,y)=1$. We say that $n=x+y$ is the {\em level} of the basic block $B_\xi(x,y)$. The basic block $B_\xi(x,y)$ is the initial block of length $|B_\xi(x,y)|$ of the 1-coding of the minimal path from $(0,0)$ to $(x,y)$.
The number of $a$'s in $B_\xi(x,y)$ is given by the binomial coefficient $ C(x+y-1,x-1)$, and the number of $b$'s by $ C(x+y-1,y-1)$.
If we are dealing with a fixed ordering, the subscript $\xi$ can be omitted.

For each $k \geq 1$ basic blocks $B_\xi ^k(x,y)$ on the symbols of $\mathscr P_k$ of the $k$-coding are defined in a similar way. We retain the meaning of the coordinates $x$ and $y$---the diagram is not telescoped. When $x+y=k$, for each $m=0,1,\dots , k$ there are $ C(k,m)$ paths from the root to the vertex at $(k-m,m)$. Since they are linearly ordered by $\xi$, we may label them in order as $P_{k,m}^1, \dots , P_{k,m}^{ C(k,m)}$. These labels $P_{k,m}^s$, for all $m$ and $s$, may be taken as the $2^k$ symbols of the alphabet $\mathscr P_k$. For each $m=0,1,\dots,k$ the basic block at the vertex $(k-m,m)$ is $P_{k,m}^1 \dots P_{k,m}^{ C(k,m)}$. To form the basic blocks at each level $n>k$, the level $n-1$ basic blocks are concatenated according to the ordering $\xi$, as above. Note that while for $k>1$, the corresponding alphabet for $\Sigma_k$ is not $\{a,b\}$, we can still talk about counts of $a$'s and $b$'s associated to $B^k_\xi(x,y)$ by considering the factor map $\Sigma_k\to \Sigma_1$. Not only does each basic block $B^k_{\xi}(x,y)$ determine the vertex $(x,y)$, but also the $k$-symbol counts in $B^k_{\xi}(x,y)$ determine $(x,y)$.

\section{The Pascal system with any ordering is essentially expansive}

For any $k$ and any path $\gamma$ in an arbitrary adic system, due to the structure of the diagram the $k$-coding of the orbit of $\gamma$ has a natural hierarchical structure of factorings for each {$n\geq k$} into basic blocks of level $n$.
This is evident for substitution subshifts and is especially simple (periodic) for odometers.
When an iterate of $\gamma$ under $T$ enters the cylinder determined by a minimal path from the root to a vertex $v$ at level $n$, the coding at that instant begins an occurrence of the basic block at level $n$. (In the cutting and stacking version, the orbit of $\gamma$ enters the base of a column and proceeds up the column.) In the following we continue to focus on the Pascal graph.

{
Fix an ordering $\xi$ of the Pascal graph $P$, a coding length $k \geq 1$, and a sequence $\omega_k \in \mathscr P_k^\Z$. For each $n\geq k$ the basic blocks of level $n$ are distinct (since they determine different vertices), and thus for fixed $n=x+y$, $B_\xi^k(x,y)$ determines the vertex $(x,y)$ of $P$ to which it is associated.
Suppose that, for each $n\geq k$, $\omega_k$ factors into a concatenation of basic blocks of level $n$ so that the factorizations are consistent across levels and respect the ordering, in the following sense: for all $n\geq k$, each basic block of the level $n+1$ factorization is the concatenation according to the fixed ordering of two adjacent basic blocks of the level $n$ factorization.
More precisely, suppose that for each $n \geq k$ we may write $\omega_k$ as a concatenation $\dots W_{n,-1}W_{n,0} W_{n,1} W_{n,2} \dots$ of blocks $W_{n,i}$, in such a way that
\begin{enumerate}
\item each $W_{n,i}$ is a basic block $B_\xi^k(x_{n,i},y_{n,i})$ of level $n$ with respect to $\xi$;
\item there is a function $j: \N \times \Z \to \Z$ such that $j(n,i+1)=j(n,i)+2$ for all $n$ and $i$, and, for each $n>1$ and each $i$, $W_{n,i}=W_{n-1,j(n,i)}W_{n-1,j(n,i)+1}$;
  \item if $\xi(x_{n,i},y_{n,i})=0$, then $W_{n-1,j(n,i)}=B_\xi^k(x_{n,i},y_{n,i}-1)$ and $W_{n-1,j(n,i)+1}=B_\xi^k(x_{n,i}-1,y_{n,i})$, while if $\xi(x_{n,i},y_{n,i})=1$, then $W_{n-1,j(n,i)}=B_\xi^k(x_{n,i}-1,y_{n,i})$ and $W_{n-1,j(n,i)+1}=B_\xi^k(x_{n,i},y_{n,i}-1)$.
\end{enumerate}
}

{
Such a system of factorings will be called a {\em $\xi$-consistent factoring scheme for $\omega_k$}. Note that for every $k$ the $k$-coding of every path in $X$ (even if one-sided) has a natural $\xi$-consistent factoring scheme determined by the sequence of vertices through which it passes. In principle, for some $k$ a $k$-coding of a path might have more than one consistent factoring scheme. In the following Proposition we show that each consistent factoring scheme determines a unique corresponding path. In Theorem \ref{thm:coding} and Corollary \ref{cor:recognizability} we show that the 3-coding of every path in $X \setminus X_\xi'$ has a unique factorization scheme since the 3-coding corresponds to a unique path.
A path in $X'_\xi$ (in the orbit of a maximal or minimal path) may have only a 1-sided $k$-coding. The following Proposition covers this case as well.
}

\begin{prop}\label{prop:consistentscheme}
{Suppose that an ordering $\xi$ of the Pascal graph $P$, $k \geq 1$, and a sequence $\omega \in \mathscr P_k^\Z$ with a $\xi$-consistent factoring scheme are given.
 Then there is a unique path $\gamma$ such that the $k$-coding of the orbit of $\gamma$ under the adic transformation $T_\xi$, insofar as defined, is (if necessary a truncation of) the given sequence $\omega$ and the natural factoring defined by the path coincides with the given one. If $\gamma \in X'_\xi$, $T$ can be extended so as to give a complete $\Z$-orbit which has the complete $k$-coding $\omega$.}
\end{prop}
\begin{proof}

Suppose that an ordering $\xi$ and $\omega \in \mathscr P_k^\Z$ with a $\xi$-consistent factoring scheme are given. Denote by $CB_n$ the central basic block of level $n$ in the given factorization,
by which we mean the basic block that includes the coordinate 0.
 Let $k_{n} \in \{0,1,\dots, |CB_n|-1\}$ be such that $\sigma^{-k_{n}}\omega$ has origin at the beginning of $CB_n$.
Form a path $\gamma$ which at each level $n$ is at vertex $v_n$ determined by the level $n$ central block $CB_n$, and with initial segment of length $n$ equal to $T_\xi^{k_n}$ applied to the minimal path from $(0,0)$ to $v_n$. The central symbol of $\omega$ determines the vertices of $\gamma$ at levels $0,\dots, k$. Then the $k$-coding $\omega_k(\gamma)$ of $\gamma$ equals the given sequence $\omega$ on the range of coordinates $I_n=[-k_n,|CB_n|-1-k_n]$ for every $n$. Thus the coding and factorization determine the sequence of vertices through which the path passes, and no other path can have the same coding and natural factorization.

{
If $\cup \{I_n:n=1,2,\dots\}=\Z$, then $\gamma \in X \setminus X'_\xi$ and the full coding $\omega_k(\gamma)$ equals the full sequence $\omega$. Otherwise, $\omega$ tells us how to define $T_\xi$ on a maximal or minimal path in such a way that $\omega$ is still the coding of a complete orbit.
For example, suppose that $k_n=|CB_n|-1$ for all $n$, so that for each $n$ the path $\gamma$ defined above is maximal from the root to $v_n$ and hence is in $X_{\xi,\max}$.
Apply the above procedure to $\sigma\omega$ to produce a path $\gamma_1$, and define $T_\xi\gamma=\gamma_1$. Then the right half of $\omega$ is the $k$-coding of the forward orbit of ${\gamma_1}$ under $T_\xi$. So we may regard the full sequence $\omega$ as the $k$-coding of the orbit of $\gamma$ under the extension of $T_\xi$ to part of $X'_\xi$. (Note that since the sequence $\omega \in \mathscr P_k^\Z$ is given, $|CB_n| \to \infty$, so that $\gamma \notin \{\gamma_{-\infty}, \gamma_\infty\}$.)}
\end{proof}

\begin{lem}\label{block containment} Fix an ordering $\xi$ of the Pascal diagram, with path space $X$. Suppose $\gamma$, $\gamma'$ are distinct paths in
$X$
 that are determined by sequences $\omega, \omega' \in \mathscr{P}_3^\Z$ with $\xi$-consistent factoring schemes
 and have full $\Z$-orbits under $T_\xi$ (extended if necessary)
 as in Proposition \ref{prop:consistentscheme}.
 Suppose further that $\gamma$ and $\gamma'$ follow only minimal edges from $(0,0)$ to distinct vertices $(x,y)$ and $(x',y')$ respectively, where $x+y=x'+y'=n$ for some $n\geq 1$. Then the 3-codings of $\gamma$ and $\gamma'$ are not the same. In other words, two such paths that
 follow only minimal edges starting at the root to
 different vertices at the same level do not have the same coding by initial path segments of length 3.

\end{lem}

\begin{proof}

Suppose $\gamma$ and $\gamma'$ are as above. In particular, we may assume that the 1- (and 2-) codings of $\gamma$ and $\gamma'$ are the same: $\omega_1(\gamma)=\omega_1(\gamma')$ and $\omega_2(\gamma)=\omega_2(\gamma')$ (otherwise we are done).
The longest minimal path from $(0,0)$ that is contained in both $\gamma$ and $\gamma'$ ends at some vertex $v=(x_0,y_0)$ at level $n_0=x_0+y_0$. In other words, $\gamma$ and $\gamma'$ \emph{split} at $v$.
If $n_0<3$ then $\gamma$ and $\gamma'$ would disagree on one of their first three edges and we would be done; we therefore consider only the cases where $n_0 \geq 3$. We may assume that no paths with a lower level splitting occur (at the same time) in the forward orbits of $\gamma$ and $\gamma'$, respectively.
Specifically, find a smallest $n_0$ for which there exists an $r$ such that $T^r\gamma$ and $T^r\gamma'$ are minimal until distinct vertices at level $n_0+1$ and contain the same edges until level $n_0$,
replace $\gamma$ and $\gamma'$ by $T^r \gamma$ and $T^r \gamma'$, and define $v$ to be the vertex where this latter pair split. It is the minimality of $n_0$ that we will exploit in the following proof by contradiction.

We consider several cases according to where in the diagram the vertex $v$ can occur. Our first case is when $v$ is an interior vertex. In other words, let $v \in U=\{(x,y):x>0$ \text{ and } $y>0\}$.
Since both paths are minimal until level $n>n_0$, there exist two minimal edges with source $v$; label $\gamma, \gamma'$ so that the left one is contained in $\gamma$ and the right one contained in $\gamma'$.
Let $k_1$ be the length of the basic block at $v$ and consider the paths $T^{k_1}\gamma$ and $T^{k_1}\gamma'$. Then at level $n_0$, $T^{k_1}\gamma$ passes through vertex $(x_0+1,y_0-1)$ and $T^{k_1}\gamma'$ passes through vertex $(x_0-1,y_0+1)$, and both are minimal until level $n_0=x_0+y_0$.
Therefore the split between the two paths in the orbits of $\gamma$ and $\gamma'$ occurred at some level $n_1$ where $n_1\leq n_0-2$. This contradicts the minimality of $n_0$. (See Figure \ref{interior}, where we indicate minimal edges by dashed lines and maximal edges by solid lines.)

\begin{figure}[h!]
\begin{center}
\begin{tikzpicture}[scale=1]
\draw[red, dashed] (8,4) -- +(1,-1);
\draw[red, dashed] (8,4) -- +(-1,-1);
\draw[red, ultra thick](6,4)--+(1,-1);
\draw[red, ultra thick](10,4)--+(-1,-1);
\fill (8,4) circle (3 pt);
\fill (9,3) circle (3 pt);
\fill (7,3) circle (3 pt);
\fill (6,4) circle (3 pt);
\fill (8,4) circle (3 pt);
\fill (6,4) circle (3 pt);
\fill (10,4) circle (3 pt);
\node[above] at (8,4) {$(x_0,y_0)$};
\node[above left] at (6,4) {$(x_0+1,y_0-1)$};
\node[above right] at (10,4) {$(x_0-1,y_0+1)$};
\node at (8.5,3.5){$\gamma'$};
\node at (7.5,3.5){$\gamma$};
\node at (9.5,3.5){$T^{k_1}\gamma'$};
\node at (6.5,3.5){$T^{k_1}\gamma$};
\end{tikzpicture}
\end{center}
\caption{$T^{k_1}\gamma$ and $T^{k_1}\gamma'$ do not agree into level $n_0=x_0+y_0$.}\label{interior}
\end{figure}
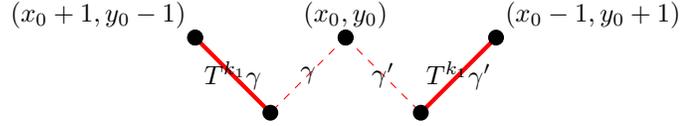

Now consider the case when $v$ is on the left side of the diagram (i.e. $v=(x_0,0)$ for some $x_0\geq 3$).
Since $\gamma, \gamma' \notin \{\gamma_{-\infty},\gamma_\infty\}$ ($\gamma_{-\infty}$ and $\gamma_\infty$ do not have a $\xi$-consistent coding scheme), both $\gamma$ and $\gamma'$ must leave the left side of the diagram.
Consider first the case when both leave the left side via minimal edges. See Figure \ref{Stepa}.
Then $n_0=x_0$ and we label so that $\gamma$ contains the minimal path to $(x_0+1,0)$ and $\gamma'$ contains the minimal path to $(x_0,1)$.
Since $\gamma$ leaves the left side of the diagram, it passes through a vertex $(x_1,1)$, and this edge is minimal.
In this case, we will examine the first edge coding of $\gamma$ and $\gamma'$. We first note that $B(x_0,1)=a^{i_0}ba^{j_0}$ where $i_0+j_0=x_0$.
(In general, the basic block at $(x,1)$ for any $x>0$ is of the form $a^{i}ba^{j}$ for some $i+j=x$, and $i$ and $j$ are nondecreasing as $x$ increases.)
Since the edge connecting vertices $(x_1,0)$ and $(x_1,1)$ is minimal, $B(x_1,1)=aB(x_1-1,1)$. Since $x_0<x_1$, $B(x_1-1,1)$ contains $B(x_0,1)$ (the blocks are the same if $x_1=x_0+1$). Therefore $B(x_1,1)=a^{i_1}ba^{j_1}$ with $i_1>i_0$ and $j_1\geq j_0$. Since $\gamma$ and $\gamma'$ are minimal into $(x_1,1)$ and $(x_0,1)$ respectively,  $ a^{i_1}ba^{j_1}$ immediately follows the decimal in $\omega_1(\gamma)$ (i.e. beginning in the origin position) and $ a^{i_0}ba^{j_0}$ immediately follows the decimal in $\omega_1(\gamma')$.
However, $i_1>i_0$ contradicts our assumption that $\gamma$ and $\gamma'$ have the same 1-coding.

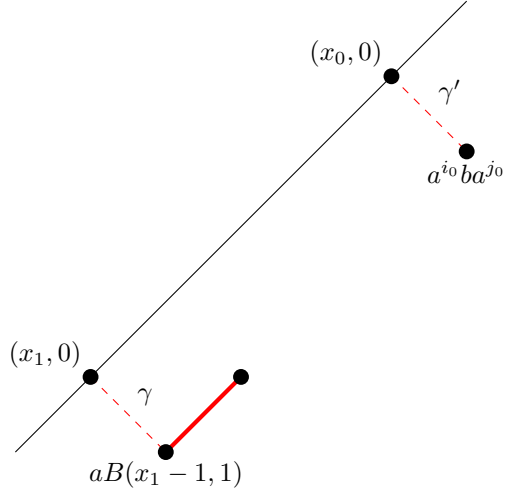
\begin{figure}[h]
\begin{center}
 \begin{tikzpicture}[scale=1]
\draw (0,0) -- +(6,6);
\draw[dashed, red](5,5)--+(1,-1);
\draw[dashed, red](1,1)--+(1,-1);
\draw[ultra thick,red](3,1)--+(-1,-1);
\fill (1,1) circle (3 pt);
\fill (5,5) circle (3 pt);
\fill (6,4) circle (3 pt);
\fill (3,1) circle (3 pt);
\fill (2,0) circle (3 pt);
\node[above left] at (5,5) {$(x_0,0)$};
\node[below] at (6,4) {$a^{i_0}ba^{j_0}$};
\node[above right] at (5.5,4.5){$\gamma'$};
\node[above left] at (1,1) {$(x_1,0)$};
\node[above right] at (1.5,.5) {$\gamma$};
\node[below] at (2,0) {$aB(x_1-1,1)$};
\end{tikzpicture}
\end{center}
\caption{The dotted lines are necessarily minimal and the thick lines are necessarily maximal.}\label{Stepa}
\end{figure}

(An analogous argument rules out the case where $\gamma$ and $\gamma'$ split along the diagonal $x=0$ and contain minimal paths from the root to $x=1$. )

 Our final and most complex case is the case when $v$ is on the left side of the diagram and $\gamma'$ leaves along a minimal edge while $\gamma$ leaves along a maximal edge at some later level.
 Thus $v=(x_0,0)$, for some $x_0\geq3$, $n_0=x_0$, $\gamma$ contains the minimal path to $(x_0+1,0)$, $\gamma'$ contains the minimal path to $(x_0,1)$ and {$\gamma$} exits the left side at $(x_1,0)$, for some $x_1>x_0$.
 If $\gamma$ is not in $X_{\xi,\max}$ it does not continue along maximal edges forever.
 Let $(u,w)$ at level $m=u+w$ be the range of the first minimal edge in $\gamma$ after $(x_1,1)$. Note that $m\geq x_1+2\geq x_0+3$. Since the edge connecting $(x_0,0)$ and $(x_0,1)$ is the first non-maximal edge in $\gamma'$, $T\gamma'$ passes through vertex $(x_0-1,1)$.
 In particular, $T\gamma'$ contains the minimal path from the root vertex into $(x_0-1,1)$. In order not to contradict the minimality of $n_0$, $T\gamma$ must also pass through vertex $(x_0-1,1)$.
 Since $\gamma$ did not pass through $(x_0,1)$, $T\gamma$ must contain the edge connecting $(x_0-1,1)$ and $(x_0-1,2)$ and this edge must be minimal according to the edge ordering. See Figure \ref{Stepb}.
 On the other hand, if $\gamma$ is in $X_{\xi,\max}$, then by assumption
 {$\gamma$} has a complete orbit determined by a complete code with a consistent factorization scheme {(see Proposition \ref{prop:consistentscheme})}. In this case, $T\gamma$ is in $X_{\xi,\min}$.
 In either case, $T\gamma$ actually consists of all minimal edges up to level {$m-1$}$\geq x_0+2$
 and therefore either passes through $(x_0-1,3)$ or $(x_0,2)$ along a minimal edge. It is this fact that we will use to arrive at a contradiction.

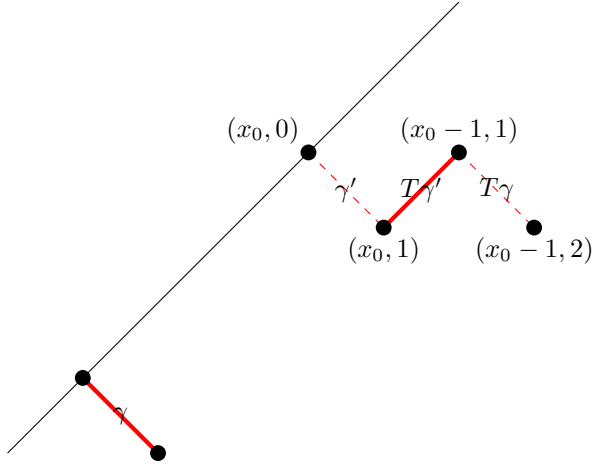
\begin{figure}[h]
\begin{center}
 \begin{tikzpicture}[scale=1]
\draw (0,0) -- +(6,6);
\draw[red, ultra thick] (1,1) -- +(1,-1);
\draw[dashed, red] (4,4)-- +(1,-1);
\draw[red,ultra thick](6,4) -- +(-1,-1);
\draw[dashed, red] (6,4)-- +(1,-1);
\fill (1,1) circle (3 pt);
\fill (4,4) circle (3 pt);
\fill (6,4) circle (3 pt);
\fill (5,3) circle (3 pt);
\fill (7,3) circle (3 pt);
\fill (2,0) circle (3 pt);
\node[above left] at (4,4) {$(x_0,0)$};
\node[below] at (5,3) {$(x_0,1)$};
\node[above] at (6,4) {$(x_0-1,1)$};
\node[below] at (7,3) {$(x_0-1,2)$};
\node at (4.5,3.5){$\gamma'$};
\node at (1.5,.5){$\gamma$};
\node at (5.5,3.5){$T\gamma'$};
\node at (6.5,3.5){$T\gamma$};
\end{tikzpicture}
\end{center}
\caption{The dotted lines are necessarily minimal and the thick red lines are necessarily maximal.}\label{Stepb}
\end{figure}

We now consider $T^{k_1}\gamma$ where $k_1=|B(x_0-1,1)|+1$, in other words, the path in the forward orbit of $T\gamma$ which first contains the maximal edge from $(x_0-2,2)$ to $(x_0-1,2)$. It contains the unique minimal path from $(0,0)$ to $(x_0-2,2)$. By an argument similar to the above, we now know that $T^{k_1}\gamma'$ must also contain the unique minimal path from $(0,0)$ to $(x_0-2,2)$ (since otherwise the orbits of $\gamma$ and $\gamma'$ would split earlier than level $n_0$), but then continue along the edge with range $(x_0-2,3)$, which is necessarily minimal. See Figure \ref{Stepc}.

\begin{figure}[h]
\begin{center}
\begin{tikzpicture}[scale=1]
\draw[red, dashed](3,1)--+(1,-1);
\draw[red, ultra thick](5,1) --+(-1,-1);
\draw[red, dashed](5,1)--+(1,-1);
\draw[red, dashed](1,1)--+(1,-1);
\draw[red, ultra thick](3,1)--+(-1,-1);
\node[above left] at (1,1) {$(x_0,0)$};
\node[above] at (3,1) {$(x_0-1,1)$};
\node[above] at (5,1) {$(x_0-2,2)$};
\node[below] at (4,0) {$(x_0-1,2)$};
\node[below] at (6,0) {$(x_0-2,3)$};
\node[left] at (3.5,.5){$T\gamma$};
\node at (4.5,.5){$T^{k_1}\gamma$};
\node[right] at (5.5, .5){$T^{k_1}\gamma'$};
\fill (1,1) circle (3 pt);
\fill(2,0) circle (3 pt);
\fill (3,1) circle (3 pt);
\fill (5,1) circle (3 pt);
\fill (4,0) circle (3 pt);
\fill (6,0) circle (3 pt);
\end{tikzpicture}
\end{center}
\caption{Step 2.}
\label{Stepc}
\end{figure}

 The issue comes at the next stage when we consider $T^{k_2}\gamma$ and $T^{k_2}\gamma'$, where $k_2=k_1+
 |B(x_0-2,2)|$. We see that $T^{k_2}\gamma'$ passes through vertex $(x_0-3,3)$. However, since the edge contained in $T^{k_1}\gamma$ leaving vertex $(x_0-1,2)$ is minimal, we see from the ordering on the edges that has been established thus far that $T^{k_2}\gamma$ must pass either through vertex $(x_0,0)$ or vertex $(x_0-2,2)$ at level $n_0=x_0$. See Figure \ref{Stepd}. Therefore we again contradict the minimality of $n_0$.

 \begin{figure}[h]
\begin{center}
\begin{tikzpicture}[scale=1]
\draw[red, dashed](3,1)--+(1,-1);
\draw[red, ultra thick](5,1) --+(-1,-1);
\draw[red, dashed](5,1)--+(1,-1);
\draw[red, dashed](1,1)--+(1,-1);
\draw[red, ultra thick](3,1)--+(-1,-1);
\draw[red, ultra thick](7,1)--+(-1,-1);
\draw[red, ultra thick](3,-1)--+(-1,1);
\draw[red, dashed](3,-1)--+(1,1);
\node[above left] at (1,1) {$(x_0,0)$};
\node[above] at (3,1) {$(x_0-1,1)$};
\node[above] at (5,1) {$(x_0-2,2)$};
\node[below] at (3,-1) {$(x_0,2)$};
\node[below] at (6,0) {$(x_0-2,3)$};
\node at (4.5,.5){$T^{k_1}\gamma$};
\node[right] at (6.5, .5){$T^{k_2}\gamma'$};
\node[left] at (1.5,.5){$T^{k_2}\gamma$};
\fill (1,1) circle (3 pt);
\fill(2,0) circle (3 pt);
\fill (3,1) circle (3 pt);
\fill (5,1) circle (3 pt);
\fill (4,0) circle (3 pt);
\fill (6,0) circle (3 pt);
\fill (7,1) circle (3 pt);
\fill (3,-1) circle (3 pt);
\end{tikzpicture}

\begin{tikzpicture}[scale=1]
\draw[red, dashed](3,1)--+(1,-1);
\draw[red, ultra thick](5,1) --+(-1,-1);
\draw[red, dashed](5,1)--+(1,-1);
\draw[red, dashed](1,1)--+(1,-1);
\draw[red, ultra thick](3,1)--+(-1,-1);
\draw[red, ultra thick](7,1)--+(-1,-1);
\draw[red, ultra thick](5,-1)--+(1,1);
\draw[red, dashed](5,-1)--+(-1,1);
\node[above left] at (1,1) {$(x_0,0)$};
\node[above] at (3,1) {$(x_0-1,1)$};
\node[above] at (5,1) {$(x_0-2,2)$};
\node[below] at (5,-1) {$(x_0-1,3)$};
\node[below] at (6,0) {$(x_0-2,3)$};
\node at (4.5,.5){$T^{k_1}\gamma$};
\node[right] at (6.5, .5){$T^{k_2}\gamma'$};
\node at (5.5,.5){$T^{k_2}\gamma$};
\fill (1,1) circle (3 pt);
\fill(2,0) circle (3 pt);
\fill (3,1) circle (3 pt);
\fill (5,1) circle (3 pt);
\fill (4,0) circle (3 pt);
\fill (6,0) circle (3 pt);
\fill (7,1) circle (3 pt);
\fill (5,-1) circle (3 pt);
\end{tikzpicture}
\end{center}
\caption{The two possibilities for $T^{k_2}\gamma$}
\label{Stepd}
\end{figure}
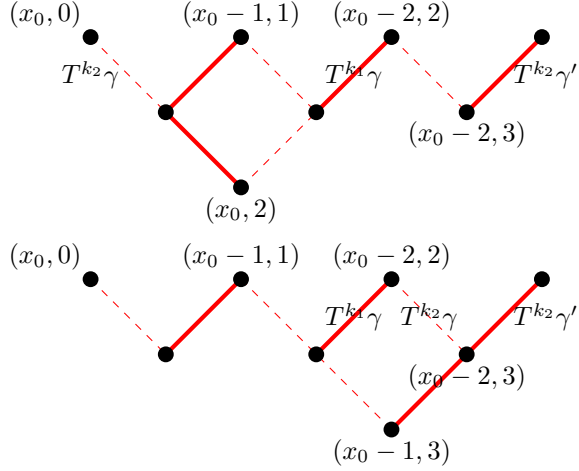

 \end{proof}

\begin{rem} \label{rem:max} Replacing $T$ by $T^{-1}$ in the above argument shows that two paths that are maximal until different vertices at the same level also do not have the same coding by initial segments of length 3.
\end{rem}

\begin{thm}\label{thm:coding}
Fix an ordering $\xi$ of the Pascal diagram, with path space $X$. Suppose $\gamma$, $\gamma'$ are distinct paths in
$X$
 that are determined by sequences $\omega, \omega' \in \mathscr{P}_k^\Z$ with $\xi$-consistent factoring schemes
and have full $\Z$-orbits under $T_\xi$ (extended if necessary) as in Proposition \ref{prop:consistentscheme}.
Then the 3-codings of $\gamma$ and $\gamma'$ are not the same. In other words, for such paths the 3-coding is faithful.
\end{thm}

\begin{proof}
Suppose, to the contrary, that there are two such distinct paths $\gamma$ and $\gamma'$ whose codings by initial segments of length 3 are the same. There exists an $r\geq{0}$ such that $T^{-r}\gamma$ and $T^{-r}\gamma'$ share the same minimal path until vertex $v=(x_0,y_0)$, for some $x_0,y_0$ with $n_0=x_0 + y_0 \geq 3$, and then follow different edges out of that vertex. Replacing $\gamma$ and $\gamma'$ by $T^{-r}\gamma$ and $T^{-r}\gamma'$, we may assume that $\gamma$ and $\gamma'$ have this property. If $v$ is on the left or right side of the diagram, then $\gamma$ and $\gamma'$ are minimal (or maximal) to different vertices, and so, by Lemma \ref{block containment} (or Remark \ref{rem:max} respectively), this cannot happen. Thus $x_0, y_0 >0$; in other words, $v$ is an ``internal'' vertex.

If the first edges after $v$ in $\gamma$ and $\gamma'$ are both maximal, then either $T^{-1}\gamma$ contains a maximal path from $(0,0)$ to $(x_0+1,y_0-1)$ and $T^{-1}\gamma'$ contains a maximal path to $(x_0-1,y_0+1)$, or vice versa. It then follows from Remark \ref{rem:max} that the 3-codes for $T^{-1}\gamma$ and $T^{-1}\gamma'$ (and hence for $\gamma$ and $\gamma'$) cannot be the same. If the first edges after $v$ in $\gamma$ and $\gamma'$ are both minimal, then by Lemma \ref{block containment} $\gamma$ and $\gamma'$ have different 3-codings.

So, without loss of generality, we may assume that the first edge in $\gamma$ after $v$ is maximal and the first in $\gamma'$ is minimal.
Assume further that the coordinate axes are chosen so that $\gamma$ goes from $v$ to $(x_0+1,y_0)$ and $\gamma'$ goes from $v$ to $(x_0,y_0+1)$.

Now $T^{|B(x_0,y_0)|}\gamma'$ begins with a minimal path from $(0,0)$ to $(x_0-1,y_0+1)$. So $T^{|B(x_0,y_0)|}\gamma$ has the same property, where it is possible that $T^{|B(x_0,y_0)|}\gamma$ is determined by $\omega'$ according to Proposition \ref{prop:consistentscheme}.
(If these paths were at different vertices at level $n_0$, the Lemma would tell us that they had distinct 3-codings.) So these paths pass through the vertex $(x_0-1,y_0+1)$. Note that if $\gamma$ is not in the orbit of a maximal path, then there is a finite string of maximal edges in $\gamma$ descending from $v$ and followed by a single minimal edge from a vertex $v_1=(x_1,y_1)$ to a vertex $v_2=(x_2,y_2)$. See Figure \ref{pascal_1}.

\begin{figure}[h!]
\begin{center}
\begin{tikzpicture}[scale=1]
\draw[dashed](0,1) to (0,0) to (1,-1);
\draw(1,-1) to (2,0);
\draw(0,0)--++(-2,-2)--++(1,-1)--++(-1,-1)--++(2,-2);
\draw[dashed](0,-6)--+(1,-1);
\draw(1,-7) to (2,-6);
\draw[dashed](2,-6) --++(1,1)--++(-1,1)--++(2,2)--++(-2,2);
\fill (0,0) circle (2 pt);
\fill (-1,-1) circle (2 pt);
\fill (1,-1) circle (2 pt);
\fill (2,0) circle (2 pt);
\fill (3,-1) circle (2pt);
\fill (0,-6) circle (2pt);
\fill (1,-7) circle (2 pt);
\fill (2,-6) circle (2pt);
\node[above] at (0,0){$(x_0,y_0)$};
\node at (-.5,-.5){$\gamma$};
\node at (.5,-.5){$\gamma'$};
\node at (4,-3){$T^{|B(x_0,y_0)|}\gamma$};
\node[left] at (-1,-1){$(x_0+1,y_0)$};
\node[below] at (1,-1){$(x_0,y_0+1)$};
\node[above] at (2,0){$(x_0-1,y_0+1)$};
\node[below] at (1,-7){$v_2$};
\node[left] at (0,-6){$v_1$};
\end{tikzpicture}
\end{center}
\caption{}
\label{pascal_1}
\end{figure}

We claim that the next edge in $\gamma'$ after $(x_0,y_0+1)$ is minimal. If this were not the case, then $T^{-1}\gamma'$ would contain a maximal path from $(0,0)$ to either $(x_0-2,y_0+2)$ or $(x_0,y_0)$, whereas $T^{-1}\gamma$ would contain a maximal path to $(x_0+1,y_0-1)$.  See Figure \ref{pascal_2}.
Again, this would mean that the 3-codings for $\gamma$ and $\gamma'$ are different.
\begin{figure}[h!]
\begin{center}
\begin{tikzpicture}[scale=1.5]
\draw[dashed](0,1) to (0,0) to (1,-1);
\draw(1,-1) to (2,0);
\draw(0,0)--++(-1,-1);
\draw[dashed] (-2,0) to (0,-2);
\draw(0,-2) to (1,-1);
\node at (.5,-1.5){$\gamma'$};
\node[above] at (0,0){$(x_0,y_0)$};
\fill (-2,0) circle (2 pt);
\fill (1,-1) circle (2 pt);
\fill (0,0) circle (2 pt);
\fill (-1,-1) circle (2 pt);
\fill (2,0) circle (2 pt);
\fill (3,-1) circle (2pt);
\fill (0,-2) circle (2 pt);
\node[left]at (-2,0){$(x_0+1,y_0-1)$};
\node at (-.5,-.5){$\gamma$};
\node at (.5,-.5){$\gamma'$};
\node at (2.5,-.5){$T^{|B(x_0,y_0)|}\gamma$};
\node[left] at (-1,-1){$(x_0+1,y_0)$};
\node[below] at (1,-1){$(x_0,y_0+1)$};
\node[above] at (2,0){$(x_0-1,y_0+1)$};
\draw[dashed](2,0) to (3,-1);
\end{tikzpicture}\vskip .2in
\begin{tikzpicture}[scale=1.5]
\draw[dashed](0,1) to (0,0) to (1,-1);
\draw(1,-1) to (2,0);
\draw[dashed] (-2,0) to (-1,-1);
\draw(0,0)--++(-1,-1);
\draw(1,-1)to (2,-2);
\draw[dashed](2,-2) to (3,-1);
\draw(3,-1) to (4,0);
\fill (4,0) circle (2 pt);
\node[right] at (4,0){$(x_0-2,y_0+2)$};
\fill (0,0) circle (2 pt);
\fill (-1,-1) circle (2 pt);
\fill (1,-1) circle (2 pt);
\fill (2,0) circle (2 pt);
\fill (3,-1) circle (2pt);
\fill(2,-2) circle (2pt);
\fill(-2,0) circle (2pt);
\node[above] at (0,0){$(x_0,y_0)$};
\node[right] at (3,-1){$(x_0-1,y_0+2)$};
\node at (-.5,-.5){$\gamma$};
\node at (.5,-.5){$\gamma'$};
\node at (1.5,-1.5){$\gamma'$};
\node at (2.5,-.5){$T^{|B(x_0,y_0)|}\gamma$};
\node[left] at (-1,-1){$(x_0+1,y_0)$};
\node[below] at (.5,-1){$(x_0,y_0+1)$};
\node[above] at (2,0){$(x_0-1,y_0+1)$};
\draw[dashed](2,0) to (3,-1);
\end{tikzpicture}
\end{center}
\caption{}
\label{pascal_2}
\end{figure}

\begin{figure}[h!]
\begin{center}
\begin{tikzpicture}[scale=1.5]
\draw[dashed](0,1) to (0,0) to (1,-1);
\draw(1,-1) to (2,0);
\draw(0,0)--++(-1,-1);
\draw[dashed] (-2,0) to (-1,-1);
\draw(-1,-1) to (0,-2);
\draw[dashed](0,-2) to (1,-1);
\draw(3,-1) to (4,0);
\fill (4,0) circle (2pt);
\node[right] at (4,0){$(x_0-2,y_0+2)$};
\node[right] at (3.5,-.5){$T^{r}\gamma$};
\node at (.5,-1.5){$\gamma'$};
\node at (-1.5,-.5){$T^{r}\gamma'$};
\fill (-2,0) circle (2 pt);
\fill (1,-1) circle (2 pt);
\fill (0,0) circle (2 pt);
\fill (-1,-1) circle (2 pt);
\fill (2,0) circle (2 pt);
\fill (3,-1) circle (2pt);
\fill (0,-2) circle (2 pt);
\node[above] at (0,0){$(x_0,y_0)$};
\node[left]at (-2,0){$(x_0+1,y_0-1)$};
\node at (-.5,-.5){$\gamma$};
\node at (.5,-.5){$\gamma'$};
\node at (2.5,-.5){$T^{|B(x_0,y_0)|}\gamma$};
\node[left] at (-1,-1){$(x_0+1,y_0)$};
\node[below] at (1,-1){$(x_0,y_0+1)$};
\node[above] at (2,0){$(x_0-1,y_0+1)$};
\draw[dashed](2,0) to (3,-1);
\end{tikzpicture}\vskip .2in
\begin{tikzpicture}[scale=1.5]
\draw[dashed](0,1) to (0,0) to (1,-1);
\draw(1,-1) to (2,0);
\draw(0,0)--++(-1,-1);
\draw[dashed](1,-1)to (2,-2);
\draw(2,-2) to (3,-1);
\draw(3,-1) to (4,0);
\fill (4,0) circle (2 pt);
\node[right] at (4,0){$(x_0-2,y_0+2)$};
\fill (0,0) circle (2 pt);
\fill (-1,-1) circle (2 pt);
\fill (1,-1) circle (2 pt);
\fill (2,0) circle (2 pt);
\fill (3,-1) circle (2pt);
\fill(2,-2) circle (2pt);
\node[above] at (0,0){$(x_0,y_0)$};
\node[right] at (3,-1){$(x_0-1,y_0+2)$};
\node at (-.5,-.5){$\gamma$};
\node at (.5,-.5){$\gamma'$};
\node at (1.5,-1.5){$\gamma'$};
\node at (2.5,-.5){$T^{|B(x_0,y_0)|}\gamma$};
\node[left] at (-1,-1){$(x_0+1,y_0)$};
\node[below] at (.5,-1){$(x_0,y_0+1)$};
\node[above] at (2,0){$(x_0-1,y_0+1)$};
\draw[dashed](2,0) to (3,-1);
\node[right] at (3.5, -.5){$T^r\gamma$};
\end{tikzpicture}
\end{center}
\caption{}
\label{pascal_3}
\end{figure}

So the next edge in $\gamma'$ after $(x_0,y_0+1)$ is minimal. Letting
\begin{equation}
r={|B(x_0,y_0)|+|B(x_0-1,y_0+1)|},
\end{equation}
it follows that $T^{r}\gamma$ contains a minimal path from $(0,0)$ to $(x_0-2,y_0+2)$, whereas $T^{r}\gamma'$ contains a minimal path to either $(x_0+1,y_0-1)$ or $(x_0-1,y_0+1)$.
See Figure \ref{pascal_3}.  
This is ruled out by applying the Lemma to $T^{r}\gamma$ and $T^{r}\gamma'$ .
\end{proof}

{
\begin{cor}\label{cor:recognizability}
For each ordering $\xi$ of the Pascal graph $P$ and each path $\gamma \in X \setminus X'_\xi$, the 3-coding $\omega_3(\gamma) \in \mathscr P_k^\Z$ has a unique $\xi$-consistent factoring scheme.
\end{cor}
\begin{proof}
Fix a particular ordering $\xi$ on the Pascal graph $P$.
That for each $\gamma \in X\setminus X'_\xi$, $\omega_3(\gamma)$ has such a natural factoring scheme is clear from the definitions of $\omega_3(\gamma)$ and the basic blocks $B_\xi^3(x,y)$, as mentioned previously.
By Proposition \ref{prop:consistentscheme} there is a unique path that has the same 3-coding and natural factoring scheme. Another factoring of the 3-coding of $\gamma$ would determine a different path which would have the same 3-coding as $\gamma$, contradicting the Theorem.
\end{proof}
}

\begin{rem}\label{rem:recognizability}
The preceding Corollary helps to clarify the relationship among paths, factorizations of codings, and factorizations of basic blocks.
Suppose we have an ordering $\xi$ on \emph{any} Bratteli diagram
 and know that there is a subset $X'_\xi$ of the path space $X$ such that for each path $\gamma \in X \setminus X'_\xi$ and each $n \geq k$ the $k$-coding $\omega_k(\gamma)$ admits a unique factorization into basic blocks of level $n$.
 (Perhaps
 $X'_\xi$ has measure 0 for each fully supported ergodic invariant measure.)
 Suppose further that the basic blocks at each level are distinct.
 Then it follows that distinct paths in $X \setminus X'_\xi$ have distinct $k$-codings, since, as in the argument above, the factorization would determine the path.
 Further, each basic block $B_\xi^k (x,y)$ would have a unique factorization into basic blocks of level $n$ for each $n$ with $k\leq n < x+y$, although basic blocks of level $l<k$ might not have unique factorizations. (The central symbol specifies the path up to the root and hence a preferred factorization of the basic blocks of levels $l<k$.) Thus under such hypotheses we would have unique factorization of basic blocks of levels greater than or equal to $k$, as shown for the case of the Pascal diagram, $1$-codings, and special orderings in Lemma \ref{lem:decomp}.
\end{rem}

\section{Topological weak mixing}
Recall that for a fixed ordering $\xi$, $(\Sigma_k,\sigma)$ denotes the subshift on $2^k$ symbols generated by coding paths according to their first $k$ edges. This is a legitimate topological dynamical system---a compact metric space together with a homeomorphism.
(We use the metric $d(\omega,\omega')=1/2^{(\inf\{|n|:\omega(n) \neq \omega'(n)\})}$). By Theorem \ref{thm:coding} the coding is essentially faithful if $k\geq 3$. Denote by $\cb(\Sigma_k)$ the set of bounded complex-valued functions on $\Sigma_k$ whose set of continuity points is residual.
In order to prove that $(\Sigma_k,\sigma)$ is topologically weakly mixing, we follow the plan in \cite{Mela2002,MelaPetersen2005} to check the condition of Keynes and Robertson \cite{KeynesRobertson1969}: a topological dynamical system which is ``closed ergodic" (there is an invariant Borel probability measure on the space for which every closed invariant set has measure 0 or 1) is topologically weakly mixing if and only if every $f \in \cb(\Sigma_k)$ for which there is $\lambda \in \C$ such that $f \circ \sigma = \lambda f$ everywhere must be constant on a dense set.
For this purpose we have to go beyond the polynomial equidistribution theorem of Weyl and quote instead a theorem of F. Hahn \cite{Hahn1965}.

In our situation $(\Sigma_k,\sigma)$ is closed ergodic, since there are fully supported ergodic measures on $\Sigma_k$ (the measures $\mu_\alpha, 0<\alpha<1$, mentioned above). Our system is topologically ergodic (there exist dense orbits). We aim to prove that if $f \in \cb(\Sigma_k)$ and $\lambda \in \C$ satisfy $f \circ \sigma = \lambda f$ everywhere, then $\lambda =1$. Then it will follow from another theorem of Keynes and Robertson \cite{KeynesRobertson1969} (since then $f \circ \sigma = f$ everywhere) that $f$ is constant on a dense set, and consequently the system is topologically weakly mixing. We also use the fact \cite[Lemma 4.18]{MelaPetersen2005} that if $f \in \cb(\Sigma_k)$ and $\lambda \in \C$ satisfy $f \circ \sigma = \lambda f$ everywhere, then every point which has both its forward and backward orbit dense is a point of continuity of $f$. Another essential ingredient is the following extension of the ``Kink Lemma" of \cite{Mela2002, MelaPetersen2005} taking into account the unknown ordering of edges.

\begin{lem}\label{lem:kink}
Let $X$ denote, as above, the set of infinite directed paths on the Pascal graph, and fix an ordering $\xi$. Let $\gamma = (\gamma_r) \in X$, with $\gamma_r$ denoting the edge of $\gamma$ from level $r-1$ to level $r$, be a path to a vertex $(i,j)$, with $i>0$ and $j>0$, which continues to the vertex $(i+1,j+1)$ in such a way that the edge $\gamma_{i+j+2}$ of $\gamma$ that enters $(i+1,j+1)$ is minimal. Let $n=i+j\geq k$ and denote by $e$ the edge with source $(i,j)$ that is not equal to $\gamma_{n+1}$.
We have to consider eight cases, labeled by triples $(a_1,a_2,a_3)$, depending on whether $\gamma_{n+1}$ is maximal ($a_1=$max) or minimal ($a_1=$min), $e$ is maximal ($a_2=$max) or minimal ($a_2=$min), and $\gamma_{n+1}$ has range $(i+1,j)$ ($a_3=$LR) or $(i,j+1)$ ($a_3=$RL).
Let
\begin{equation}\label{eq:returntime}
r_n=
\begin{cases}
 C(n,j) &\text{cases (max, min, LR) and (max, min, RL)}\\
 C(n+1,j+1) &\text{cases (min, min, RL) and (max, max, LR)}\\
 C(n+1,j) &\text{cases (min, min, LR) and (max, max, RL)}\\
 C(n+1,j)+ C(n,j+1) &\text{cases (min, max, LR) and (min, max, RL).}
\end{cases}
\end{equation}
Then $T^{r_n}\gamma$ and $\gamma$ agree on their first $n$ edges. Consequently, if $x$ is a path that begins with the minimal path to $(i,j)$, $z$ is a path that begins with the maximal path to $(i,j)$, $T^m x = \gamma$, and $T^l \gamma=z$, then $d(\sigma^{r_n} \omega_k (\gamma), \omega_k (\gamma)) \leq 1/2^{\min\{m,l\}}$.
\end{lem}
\begin{proof}
Let $z$ denote the path which is maximal to the vertex $(i,j)$ and agrees with $\gamma$ after level $n=i+j$, so that $z=T^l \gamma$ for some $l \geq 0$.
Suppose first that the edge $\gamma_{n+1}$ is maximal and $e$ is minimal. Then in both cases (max, min, LR) and (max, min, RL) (whether $\gamma_{n+1}$ has range $(i+1,j)$ or $(i,j+1)$) $Tz$ begins with the minimal path to $(i,j)$. Then choosing $m \geq 0$ so that $m+l+1= C(n,j)$ yields that $T^m Tz=T^{m+l+1}\gamma$ agrees with $\gamma$ on their first $n$ edges. Then $r_n= C(n,j)$ has the stated property. See Figure \ref{fig:lem4.1}.

\begin{figure}[h]
\begin{center}
\begin{tikzpicture}[scale=.5]
\draw(0,0)--++(-2,3)--++(1,1)--++(-1,2)--++(1,1)--++(-1,1);
\draw(0,0)--++(2,2)--++(-1,1)--++(1,2)--++(-1,2)--++(1,1);
\draw(0,0)--+(-2,-2);
\draw[dashed](-2,-2)--++(2,-2);
\draw[dashed](0,0)--+(2,-2);
\draw(2,-2)--+(-2,-2);
\fill (0,0) circle (3 pt);
\fill (-2,-2) circle (3 pt);
\fill (2,-2) circle (3 pt);
\fill (0,-4) circle (3 pt);
\node[left]at (0,0){$(i,j)$};
\node[left]at (-2,-2){$(i+1,j)$};
\node[right]at (2,-2){$(i,j+1)$};
\node[below]at (0,-4){$(i+1,j+1)$};
\node at (-2.5,4) {$\gamma$};
\node at (3,4) {$z=T^l\gamma$};
\node[left] at (-1,-1){$\gamma_{n+1}$};
\node[right] at (1,-1){$e$};
\end{tikzpicture}
\end{center}
\caption{Case (max, min, LR)}
\label{fig:lem4.1}
\end{figure}
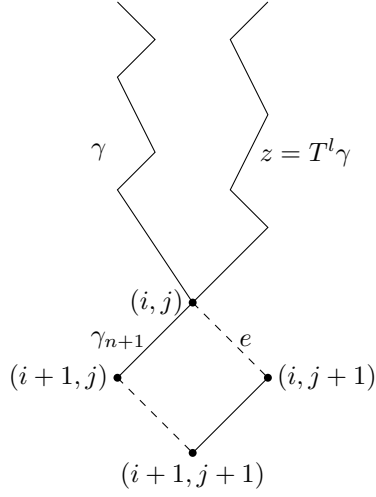

The other cases are similar but differ in the details, and the particular values of $r_n$ that we find are important for the following theorem, so we go through them all.
Suppose now that the edge $e$ is maximal and $\gamma_{n+1}$ is maximal and has range $(i+1,j)$ (case (max, max, LR)). In this case $Tz$ begins with the minimal path to the vertex $(i-1,j+1)$ followed by the minimal edge to $(i,j+1)$ and the maximal edge to $(i+1,j+1)$. Then $u=T^{ C(n,j+1)-1} Tz$ begins with the maximal path to $(i-1,j+1)$, and $y=Tu$ begins with the minimal path to $(i,j)$ followed by the maximal edge to $(i,j+1)$. Choose $m \geq 0$ so that $m+l+1= C(n,j)$. Then $T^m Tu$ agrees with $\gamma$ on their first $n$ edges. We have $T^m Tu=T^m T^{ C(n,j+1)} Tz=T^{m+l+1+ C(n,j+1)}\gamma=T^{ C(n,j)+ C(n,j+1)}\gamma=T^{ C(n+1,j+1)}\gamma$, so $r_n= C(n+1,j+1)$ acts as stated.

If the edge $e$ is maximal and $\gamma_{n+1}$ is maximal and has range $(i,j+1)$ (case (max, max, RL)), the argument is the same as the immediately preceding one, except that $j$ is replaced by $j-1$. Namely, $Tz$ begins with the minimal path to $(i+1,j-1)$, $u=T^{ C(n,j-1)-1}Tz$ begins with the maximal path to $(i+1,j-1)$, $y=Tu$ begins with the minimal path to $(i,j)$ followed by the maximal edge to $(i+1,j)$. Again choose $m \geq 0$ so that $m+l+1= C(n,j)$. Then $T^m Tu$ and $\gamma$ agree on their first $n$ edges,
so we obtain $r_n= C(n,j)+ C(n,j-1)= C(n+1,j)$.

Consider now the case (min, max, LR). Then $Tz$ begins with the minimal path to $(i+1,j-1)$, $u=T^{ C(n,j-1)-1}Tz$ begins with the maximal path to $(i+1,j-1)$, and $Tu$ begins with the minimal path to $(i-1,j+1)$ followed by a minimal and then a maximal edge. Then $y=TT^{ C(n,j+1)-1}Tu$ begins with the minimal path to $(i,j)$ extended by two maximal edges. Thus choosing $m \geq 0$ so that $m+l+1= C(n,j)$ yields that $T^my = T^{m+l+1+ C(n,j+1)+ C(n,j-1)}\gamma=T^{ C(n+1,j)+ C(n,j+1)}\gamma$ agrees with $\gamma$ on the first $n$ edges..

Next, consider the case (min, max, RL). Then $Tz$ is minimal to $(i-1,j+1)$, $u=T^{ C(n,j+1)-1}Tz$ is maximal to $(i-1,j+1)$, and $p=Tu$ is minimal to $(i+1,j)$ extended by a maximal edge.
Then $T^{ C(n,j-1)-1}p$ is maximal to $(i+1,j-1)$, so that $y=T^{ C(n,j-1)}p$ is minimal to $(i,j)$.
Thus if we choose $m$ so that $m+l+1= C(n,j)$, we will have $T^my=T^mT^{ C(n,j-1)}p=T^{m+ C(n,j-1)+ C(n,j+1)+l+1}\gamma = T^{ C(n+1,j)+ C(n,j+1)}\gamma$ and $\gamma$ agreeing on their first $n$ edges.

Now consider the case (min, min, LR). Then $Tz$ starts with the minimal path to $(i+1,j-1)$, $u=T^{ C(n,j-1)-1}Tz$ starts with the maximal path to $(i+1,j-1)$, $y=Tu$ is minimal to $(i,j)$ followed by a minimal and then a maximal path; and if we choose $m$ so that $m+l+1= C(n,j)$, then $T^my=T^{ C(n+1,j)}\gamma$ and $\gamma$ agree on the first $n$ edges.

The final case is (min, min, RL). Then $Tz$ is minimal to $(i-1,j+1)$ and extends from there by a maximal edge followed by a minimal edge. Then $u=T^{ C(n,j+1)-1}Tz$ is maximal to $(i-1,j+1)$, from where it extends with a maximal and then a minimal edge. Thus $y=Tu$ is minimal to $(i,j)$ and extends by the minimal edge $e$ followed by a maximal edge; and, if we choose $m$ so that $m+l+1= C(n,j)$, then $T^my=T^{m+l+1+ C(n,j+1)}=T^{ C(n+1,j+1)}\gamma$ and $\gamma$ agree on their first $n$ edges.

The final statement in the Lemma follows from the fact that $T^i (T^{r_n}\gamma)$ and $T^i \gamma$ agree on their first $n\geq k$ edges for $-m \leq i \leq l$.
\end{proof}

\begin{thm}\label{thm:weakmixing}
For each ordering and each $k \geq 1$, the subshift $(\Sigma_k,\sigma)$ is topologically weakly mixing.
\end{thm}
\begin{proof}
Fix an ordering $\xi$.
We deal only with paths in $X$ whose orbit does not contain any paths that are eventually diagonal or eventually consist of only maximal or only minimal edges.
By Proposition \ref{prop:maxpaths} the set of such paths has measure 1 for each ergodic $T$-invariant probability measure on $X$. The coding $\omega_k(\gamma)$ of any such path $\gamma$ has both its forward and backward orbit dense in $\Sigma_k$ and hence, as noted above, is a point of continuity of any eigenfunction $f \in \cb (\Sigma_k)$.

Suppose that $f \in \cb (\Sigma_k), \lambda \in \C$, and $f \circ \sigma = \lambda f$ everywhere on $\Sigma_k$;
we want to show that then $\lambda=1$. Fix a point $\omega_k(\gamma)$ that is the coding of a path $\gamma$ as described above and hence is a point of continuity of $f$.
For each $n$ let $(i_n,j_n)$ denote the vertex through which $\gamma$ passes at level $n$, i.e. the range of the edge $\gamma_n$. For each $n$, the path $\gamma$ can be modified if necessary below level $n$ to a path $\gamma'$ that satisfies the hypotheses of Lemma \ref{lem:kink}, is not eventually diagonal, and does not eventually consist of only maximal or minimal edges.
This determines $r_n,m_n,l_n$ so that $T^{r_n}\gamma', \gamma'$, and $\gamma$ agree on their first $n$ edges, and $d(\sigma^{r_n} \omega_k (\gamma'), \omega_k (\gamma')) \leq 1/2^{\min\{m_n,l_n\}}$. Since $m_n$ and $l_n$ are unbounded this implies that $d(\sigma^{r_n} \omega_k(\gamma'), \omega_k(\gamma')) \to 0$ as $n\to \infty$. Since $f(\sigma ^{r_n} \omega_k(\gamma')) = \lambda ^{r_n} f(\omega_k(\gamma'))$ for all $n$, we have $\lambda^{r_n} \to 1$ as $n \to \infty$. In particular $|\lambda|=1$.

Suppose that $\lambda$ is a root of unity, so that $\lambda=\exp (2 \pi i a/b)$ for two relatively prime integers $a$ and $b$.
Let $q$ be a prime that divides $b$. In the Pascal triangle we can find infinitely many $n$ for which there is a $j$ such that none of $ C(n,j),  C(n+1,j+1)$, $ C(n+1,j)$, $ C(n+1,j)+ C(n,j+1)$ is divisible by $q$. For example, by Lemma 5.1 of \cite{AdamsPetersen1998} $ C(q^s-2,k)\equiv_q (-1)^{k}(k+1)$ for $k=0\dots, q^s-2$.
Therefore, for each $s$ the entire row $ C(q^s-1,j), 0 \leq j \leq q^s-1$, is not divisible by $q$, and in the preceding row there are many entries $ C(q^s-2,tq-1), 0 < t \leq q^{s-2}$, that are not divisible by $q$ as well as many that are, so choosing $n=q^s-2$
and $j$ so that $j+1$ is not divisible by $q$ but $j+2$ is divisible by $q$ will do. (If $j+1$ is not divisible by $q$, then neither is $ C(q^s-2,j)$, while if $j+2$ is divisible by $q$, then so is $ C(q^s-2,j+1)$.) We can arrange that our path $\gamma$ hits such vertices for infinitely many $n$, and then $|\lambda^{r_n}-1| \geq 1/q$ for all of these values of $n$, contradicting that $\lambda^{r_n} \to 1$. Thus there are no roots of unity besides $1$ among the eigenvalues of $\sigma$.

Suppose now that $\lambda$ is not a root of unity, say $\lambda = \exp (2 \pi i \beta)$ for an irrational $\beta \in (0,1)$.
According to the theorem of F. Hahn \cite[Theorem 5]{Hahn1965} the transformation $S(x_1,x_2,x_3,\dots)=(x_1+\beta,x_2+x_1,x_3+x_2,\dots)$ (mod $1$) on the infinite torus is minimal, in fact strictly ergodic, with Haar measure the unique invariant probability measure. Then the orbit of $(0,0,0,\dots)$ is dense, and $S^n(0,0,0,\dots)=( C(n,1)\beta, C(n,2)\beta,\dots, C(n,n)\beta,0,0,\dots) \mod 1$.
We use this result to construct a continuity point $\omega_k(\gamma)$ of $f$ that is the coding of a path $\gamma$ such that, denoting by $(i_n,j_n)$ the vertex through which $\gamma$ passes at level $n$, there are infinitely many $n$ for which all of $ C(n,j_n)\beta \mod 1,  C(n+1,j_n+1) \beta \mod 1,  C(n+1,j_n) \beta \mod 1$, and $[ C(n+1,j)+ C(n,j+1)]\beta \mod 1$ are a fixed distance from $0$ and $1$, again making it impossible that $\lambda^{r_n} \to 1$.
Given an initial segment of $\gamma$ ending at a vertex $(i_0,j_0)$, extend $\gamma$ several steps so as to use both some maximal and some minimal edges as well as some edges that increase the first ($i$) coordinate and the second ($j$) coordinate, arriving at a vertex $(i_1,j_1)$.
Then add edges that increase the $i$ coordinate until we arrive at a vertex $(i,j)$ for which $( C(i+j,j-1), C(i+j,j), C(i+j,j+1))\beta \approx (1/8,1/8,1/8) \mod 1$ (say $\{ C(i+j,j-1) \beta\} - 1/8| < 1/16$, etc., with $\{x\}$ denoting $x \mod 1$).
Then $ C(i+j+1,j)\beta \mod 1$ and $ C(i+j+1,j+1)\beta \mod 1$ are in the interval $(1/8,3/8)$,
while $[ C(i+j+1,j)+ C(i+j,j+1)]\beta \mod 1$ is in the interval $(3/16,9/16)$. Continue, making it impossible that $\lambda^{r_n} \to 1$.
\end{proof}

\section{Condition for topological conjugacy to an odometer}

According to the theorem of Downarowicz and Maass \cite{DownarowiczMaass2008} mentioned above, every minimal adic system with finite topological rank is either expansive or topologically conjugate to an odometer. For systems with a bounded number of vertices per level some sufficient conditions for topological conjugacy to an odometer are known (see the remarks in the introduction). In this section we give a necessary and sufficient condition for a fairly general adic system to be topologically conjugate to an odometer. The condition is then used to determine when, with probability 1, an adic system with a random ordering, or an adic system constructed according to a random process, is topologically conjugate to an odometer.

In this section we deal with an arbitrary Bratteli diagram $B=(\mathscr V,\mathscr E)$. As usual the vertices are arranged into levels, so that $\mathscr V = \cup_{i \geq 0} \mathscr V_i$, $\mathscr E = \cup_{i \geq 1} \mathscr E_i$, and there are directed edges from each level to the next: there are a range map $r$ and a source map $s$ from $\mathscr E$ to $\mathscr V$ such that $r(\mathscr E_i)=\mathscr V_i$ for all $i \geq 1$, $s(\mathscr E_i)=\mathscr V_{i-1}$ for all $i \geq 1$, $s^{-1}v\neq \emptyset$ for all $v\in \mathscr V$, and $r^{-1}v'\neq \emptyset$ for all $v' \in \mathscr V\setminus \mathscr V_0$.
We assume that $\mathscr V_0$ consists of a single vertex, $v_0$, and $|\mathscr V_n|, |\mathscr E_n| < \infty$ for all $n$.

The {\em path space} $X$ is the space of infinite paths $x=x_1x_2\dots$ such that $x_i \in \mathscr E_i$ and $r(x_i)=s(x_{i+1})$ for all $i$. To avoid degenerate cases, we assume that $X$ with the usual metric (in which two infinite paths are close if they agree on long initial segments)
 is homeomorphic to the Cantor set. An {\em ordering} $\xi$ of the diagram is an assignment of a linear order to the set $r^{-1}v$ of edges entering each vertex $v \in \mathscr V$. $X_{\xi, \max}$ consists of the set of paths all of whose edges are maximal in the ordering $\xi$, and $X_{\xi, \min}$ of all the paths all of whose edges are minimal. As with the Pascal diagram, the ordering $\xi$ determines the adic transformation $T_\xi : X \setminus X_{\xi,\max} \to X \setminus X_{\xi, \min}$, which is a homeomorphism.

See \cite{HermanPutnamSkau1992}, \cite{GiordanoPutnamSkau1995}, or \cite{Durand2010} for the definitions of {\em telescoping} and {\em isomorphism} for diagrams and ordered diagrams. {\em Equivalence} of diagrams or ordered diagrams is the equivalence relation generated by isomorphism and telescoping. Two diagrams $B_1$ and $B_2$ are diagram equivalent if and only if there is a diagram $B$ such that telescoping $B$ to odd levels yields a telescoping of one of $B_1, B_2$ and telescoping $B$ to even levels yields a telescoping of the other one; similarly for ordered diagrams. If two ordered diagrams are equivalent, then the topological dynamical systems that they define by means of their adic transformations are topologically conjugate. Recall that a Bratteli diagram is called \emph{simple} if there is a telescoping for which every pair of vertices $v_n\in \mathscr{V}_n$ and $v_(n+1)\in \mathscr{V}_{n+1}$ there is an edge between $v_n$ and $v_{n+1}$.
That the converse holds also, for simple adic systems that are properly ordered (with unique maximal and minimal paths), follows from Theorem 4.4 of \cite{HermanPutnamSkau1992}. (For suppose that simple and properly ordered diagrams $D_1$ and $D_2$ determine adic systems $X_1$ and $X_2$, respectively, that are topologically conjugate under an adic-commuting homeomorphism $h: X_1 \to X_2$. Each $X_i$ has a natural sequence of Kakutani-Rokhlin partitions $\mathscr P_i$ corresponding to its diagram structure. Then $h^{-1} \mathscr P_2$ is a sequence of Kakutani-Rokhlin partitions for $X_1$ that is admissible as in Theorem 4.2 of \cite{HermanPutnamSkau1992}. This sequence determines $D_2$ according to the process in \cite{HermanPutnamSkau1992}, so by Theorem 4.4 of \cite{HermanPutnamSkau1992} $D_1$ and $D_2$ are diagram equivalent.) (Note that in the proof of Theorem 4.4 of \cite{HermanPutnamSkau1992} the tops $Z_n$ and $\tilde{Z}_n$ of the tower diagrams have the property that given $n$ we can find $m$ such that $\tilde{Z}_m$ is contained in $Z_n$, and vice versa.)

An {\em odometer} is any topological dynamical system that is topologically conjugate to the system defined by the adic transformation on the path space of an ordered diagram which has exactly one vertex at every level.

\begin{figure}[h!]
\begin{center}
\begin{tikzpicture}[scale=2.5]
\draw[rounded corners] (0,1) --(1.4,.4) -- (3,0);
\draw[rounded corners] (0,1) -- (1.6,.6) -- (3,0);
\draw (0,1) -- (1,0);
\draw[rounded corners] (2,1) --(2.4,.4) -- (3,0);
\draw[rounded corners] (2,1) -- (2.6,.6) -- (3,0);
\draw (2,1) -- (1,0);
\draw[rounded corners] (4,1) --(3.6,.4) -- (3,0);
\draw[rounded corners] (4,1) -- (3.4,.6) -- (3,0);
\draw (4,1) -- (1,0);
\fill (0,1) circle (1 pt);
\fill (2,1) circle (1 pt);
\fill (4,1) circle (1 pt);
\fill (1,0) circle (1 pt);
\fill (3,0) circle (1 pt);
\node at (.8,.2){2};
\node at (1.2,.2){1};
\node at (1.5,.15){3};
\node at (2.2,.2){2};
\node at (2.5,.2){5};
\node at (2.65,.25){1};
\node at (2.85,.225){4};
\node at (3.15,.25){3};
\node at (3.4,.25){6};
\node[above] at (0,1){$v_1$};
\node[above] at (2,1){$v_2$};
\node[above] at (4,1){$v_3$};
\node[below] at (1,0){$v_2v_1v_3$};
\node[below] at (3,0){$(v_2v_1v_3)^2$};
\end{tikzpicture}
\end{center}
\caption{A uniformly ordered level, where $\omega=v_2v_1v_3$.}\label{UniformlyOrdered}
\end{figure}
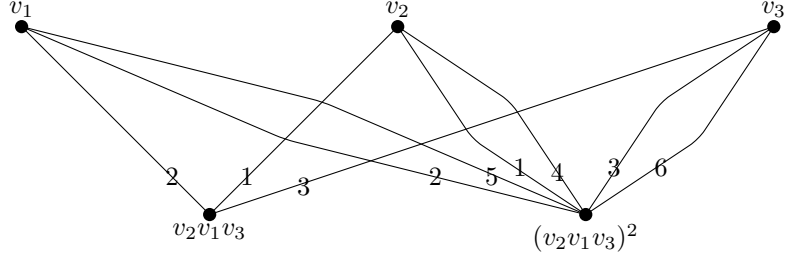

For each $n$ and each $w\in \mathscr V_n$, there is an associated coding in terms of vertices in $\mathscr V_{n-1}$. (This is the ``morphism read on $\mathscr V_n$'' in \cite[p. 342]{Durand2010}.) More precisely, if $r^{-1}w=\{e_1,e_2,\dots,e_{|r^{-1}w|}\}$, where the subscripts correspond directly to the ordering given by $\xi$, then the coding of $w$ by vertices in $\mathscr V_{n-1}$ is given by $c(w)=s(e_1)s(e_2)\dots s(e_{|r^{-1}w}|)$. We call $\mathscr E_{n}$ and the level $n$ \emph{uniformly ordered} if there is a string $v = v_{i_1}v_{i_2}\dots v_{i_k}$, with each $v_{i_m}$ in $V_{n-1}$, such that for every $w$ in $\mathscr V_{n}$, $c(w)=v^{k_w}$ for some $k_w\in \N$. See Figure \ref{UniformlyOrdered}. Note that if any level has only one vertex, then it is automatically uniformly ordered.

For a Bratteli-Vershik system $B=(\mathscr V,\mathscr E)$, define an \emph{ordered shape at level $n$} to consist of $\mathscr V_{n-1}$, $\mathscr V_{n}$, $\mathscr E_{n}$, and a partial order on $\mathscr E_{n}$ which is a total order on each $r^{-1}w, w \in \mathscr V_n$. We say that a shape is \emph{uniformly ordered} if $\mathscr E_{n}$ is uniformly ordered.

\begin{lem}\label{TPUO} If a diagram $B=(\mathscr V,\mathscr E)$ with an ordering $\xi$ has infinitely many levels for which $\mathscr E_n$ is uniformly ordered, then (1) there is a telescoping $B'=(\mathscr V', \mathscr E')$ of $B=(\mathscr V, \mathscr E)$ for which every level is uniformly ordered, and (2) $B=(\mathscr V,\mathscr E)$ is simple and properly ordered.
\end{lem}

\begin{proof} First note that if an ordered Bratteli diagram is telescoped from level $j_1$ to level $j_2$, the new coding of the vertices at level $j_2$ in terms of the vertices at level $j_1$ results from decomposing the original $j_2-1$ coding of vertices in $\mathscr V_{j_2}$ by vertices in $\mathscr V_{j_2-1}$ in an inductive manner into a level $j_1$ coding. Then note that if $j_1<n\leq j_2$ and $\mathscr E_n$ is uniformly ordered, then the new set of edges arising from telescoping from level $j_1$ to level $j_2$ will be uniformly ordered. Indeed, since all vertices $w\in \mathscr V_n$ have codings $c(w)=v^{k_w}$ in terms of the vertices at level $n-1$, decomposing the level $j_2-1$ codings of vertices in $\mathscr V_{j_2}$ into level $n-1$ codings will lead to repeated concatenations of $v$. Furthermore, the ordering $\xi$ will lead to a unique decomposition of $v$ in terms of vertices at level $j_1$, and therefore the new set of edges will be uniformly ordered. See Figure \ref{everylevelUO}.

In order for a level to be uniformly ordered, every vertex in $\mathscr{V}_n$ must connect to every vertex in $\mathscr{V}_{n+1}$. Hence any diagram with infinitely many uniformly ordered levels is simple.

For every $n$ there is a unique vertex $v_{n-1}$ in $\mathscr{V}_{n-1}$ such that every edge in $\mathscr{E}_n$ which is maximal has $v_{n-1}$ as its source. Therefore any path in $X_{\max}$ must go through these vertices, and there is a unique path which travels along maximal edges at every level.
\end{proof}

\begin{figure}[h!]
\begin{center}
\begin{tikzpicture}[scale=2.5]
\draw[rounded corners] (0,1) --(1.4,.4) -- (3,0);
\draw[rounded corners] (0,1) -- (1.6,.6) -- (3,0);
\draw (0,1) -- (1,0);
\draw[rounded corners] (2,1) --(2.4,.4) -- (3,0);
\draw[rounded corners] (2,1) -- (2.6,.6) -- (3,0);
\draw (2,1) -- (1,0);
\draw[rounded corners] (4,1) --(3.6,.4) -- (3,0);
\draw[rounded corners] (4,1) -- (3.4,.6) -- (3,0);
\draw (4,1) -- (1,0);
\draw (1,2) -- (0,1);
\draw (1,2) -- (2,1);
\draw (3,2) -- (2,1);
\draw (3,2) -- (4,1);
\fill (1,2) circle (1 pt);
\fill (3,2) circle (1 pt);
\fill (0,1) circle (1 pt);
\fill (2,1) circle (1 pt);
\fill (4,1) circle (1 pt);
\fill (1,0) circle (1 pt);
\fill (3,0) circle (1 pt);
\node at (.8,.2){2};
\node at (1.2,.2){1};
\node at (1.5,.15){3};
\node at (2.2,.2){2};
\node at (2.5,.2){5};
\node at (2.65,.25){1};
\node at (2.85,.225){4};
\node at (3.15,.25){3};
\node at (3.4,.25){6};
\node[above] at (1,2){$a$};
\node[above] at (3,2){$b$};
\node[above left] at (0,1){$a$};
\node[above right] at (2,1){$ab$};
\node[above right] at (4,1){$b$};
\node[below] at (1,0){$abab$};
\node[below] at (3,0){$(abab)^2$};
\end{tikzpicture}
\end{center}
\caption{Following the code shows that it is possible to telescope so that every level is uniformly ordered.}\label{everylevelUO}
\end{figure}
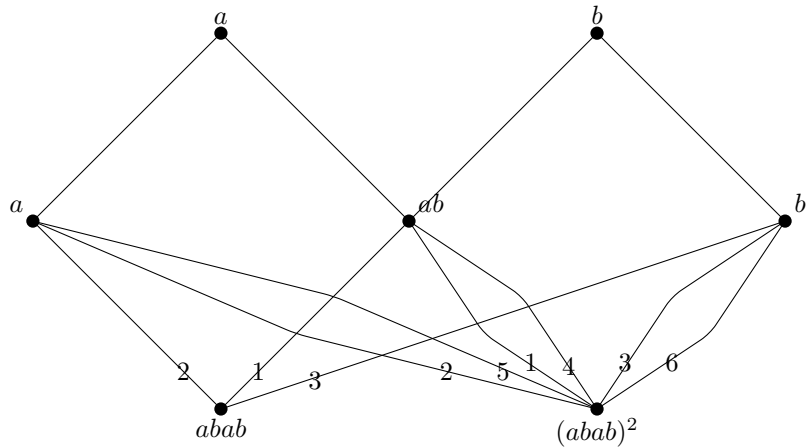

\begin{rem} It is possible for two non-uniformly ordered levels to telescope to one that is uniformly ordered. Indeed, suppose that $|\mathscr V_{n}|\geq 3$ and the level $n-1$ coding of vertices at level $n$ is such that $c(w_2)=c(w_3)\neq c(w_1)$ for $w_1,w_2,w_3$ in $\mathscr V_{n}$. Then suppose that the level $n$ coding of vertices in $\mathscr V_{n+1}$ is such that the only difference in the codings of different vertices is that $w_2$ and $w_3$ are interchanged in two codings. Then neither $\mathscr E_n$ nor $\mathscr E_{n+1}$ is uniformly ordered. However, once one telescopes between levels $n-1$ and $n+1$, the new set of edges will be uniformly ordered. See Figure \ref{NUO}\end{rem}.

 \begin{figure}[h!]
\begin{center}
\begin{tikzpicture}[scale=1]
\draw (0,2) -- (0,1);
\draw (2,2) -- (0,1) -- (1,0);
\draw (0,1) -- (3,0);
\draw (4,2) -- (2,1)-- (1,0);
\draw (4,2) -- (4,1);
\draw (2,2) -- (2,1);
\draw (2,2) -- (4,1) -- (3,0);
\draw (0,1) -- (1,0);
\fill (0,2) circle (2 pt);
\fill (2,2) circle (2 pt);
\fill (4,2) circle (2 pt);
\fill (0,1) circle (2 pt);
\fill (2,1) circle (2 pt);
\fill (4,1) circle (2 pt);
\fill (1,0) circle (2 pt);
\fill (3,0) circle (2 pt);
\node[above] at (0,2){$v_1$};
\node[above] at (2,2){$v_2$};
\node[above] at (4,2){$v_3$};
\node[left] at (0,1){$w_1=v_1v_2$};
\node[below] at (2,1){$w_2=v_2v_3$};
\node[right] at (4,1){$w_3=v_2v_3$};
\node[below left] at (1,0){$w_1w_2=v_1v_2v_2v_3$};
\node[below right] at (3,0){$w_1w_3=v_1v_2v_2v_3$};
\end{tikzpicture}
\end{center}
\caption{All orderings are left to right. Neither level is uniformly ordered, but the telescoping is.}\label{NUO}
\end{figure}
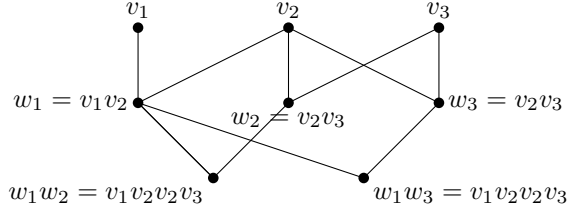

\begin{thm}\label{UOimpliesOdometer}
 If an ordered Bratteli diagram $B=(\mathscr V,\mathscr E,\xi)$ {has a telescoping for which there are} infinitely many uniformly ordered levels then the associated adic system is topologically conjugate to an odometer. Conversely, any simple properly ordered Bratteli diagram $B=(\mathscr V,\mathscr E, \xi)$ that is topologically conjugate to an odometer has a telescoping for which there are infinitely many uniformly ordered levels.
\end{thm}
\begin{proof}
Suppose that we have a diagram $B=(\mathscr V, \mathscr E, \xi)$ {for which there is a telescoping} with infinitely many uniformly ordered levels.
By Lemma \ref{TPUO} we may assume that the diagram is uniformly ordered at every level $n>1$; also note that the diagram is simple and properly ordered. One can see then that all the $k$-codings (coding $T_\xi$-orbits according to the initial path segments of length $k$) are periodic, so that the system is not expansive and hence, by \cite{DownarowiczMaass2008}, is topologically conjugate to an odometer---in case the system has finite topological rank. We give here an explicit proof in the general case by constructing an intermediate ordered Bratteli diagram, $B'$, with the property that telescoping to odd levels yields the given diagram, $B$, and telescoping to even levels yields a system with a single vertex at every level. Thus $B$ is equivalent to the diagram of an odometer, and hence its associated topological dynamical system is topologically conjugate to an odometer.

    To construct $B'$, for each $n \geq 1$ let $|{\mathscr V'_{2n-1}}|= |{\mathscr V_{n}}|$ and $|{\mathscr V'_{2n}}|=1$. Let $\mathscr E'_1=\mathscr E_1$. For every $n \geq 1$, denote the single vertex in $\mathscr V'_{2n}$ by $w'_{2n}$ and set up a one-to-one correspondence between $\mathscr V'_{2n-1}$ and $\mathscr V_n$. Fix $n \geq 1$. {Since $B$ is uniformly ordered at level $n+1$, there is a string $v=v_1 \dots v_{|v|}$ with entries from $\mathscr V_{n}$ such that for each $w$ in $\mathscr V_{n+1}$ the coding of $w$ by vertices in $\mathscr V_{n}$ is $c(w)=v^{k_w}$.}
    To define $\mathscr{E}'_{2n}$ let $|r^{-1}w'_{2n}|=|v|$,
      label the edges $e_i'$ in $r^{-1}w'_{2n}$ so that for each $i$ $s(e_i')$ is the vertex $u'$ in $\mathscr V'_{2n-1}$ which corresponds to $v_i$, and order these edges according to their labels.
       To define $\mathscr{E}'_{2n+1}$, note that each $u'\in \mathscr V'_{2n+1}$ corresponds to a $w\in \mathscr V_{n+1}$; we let the number of edges from $w'_{2n}$ to $u'$ be $k_w$ for the $w$ corresponding to $u'$ and order these edges arbitrarily.
       Then telescoping $B'$ to odd levels yields $B$, and telescoping to even levels yields a system with just one vertex at every level, an odometer. See Figure \ref{ID}.

Conversely, suppose that $B=(\mathscr V, \mathscr E, \xi)$ is a simple properly ordered diagram whose associated adic system is topologically conjugate to an odometer. As mentioned above, then the diagram is equivalent to an odometer:
 there is an intermediate diagram $B'=(\mathscr V',\mathscr E',\xi')$, which when telescoped to odd levels yields a telescoping $B''$ of $B$ and when telescoped to even levels yields a system with just one vertex at each level, or vice versa.
  For each $n$, $|\mathscr V'_{2n}|=1$ and hence level $2n$ of $B'$ is automatically uniformly ordered. Then as in the proof of Lemma \ref{TPUO}, the telescoping of $B'$ from level $2n-1$ to level $2n+1$ will yield a set of edges whose ranges form a uniformly ordered level. Therefore in $B''$ every level $\mathscr E_n$ must be uniformly ordered.
\end{proof}

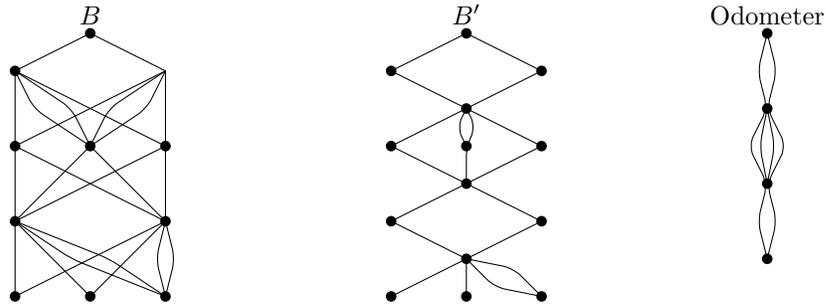
\begin{figure}[h!]
\begin{center}
\begin{tikzpicture}[scale=.5]
\fill (0,0) circle (4 pt);
\fill (0,2) circle (4 pt);
\fill (0,4) circle (4 pt);
\fill (0,6) circle (4 pt);
\fill (2,0) circle (4 pt);
\fill (2,4) circle (4 pt);
\fill (2,7) circle (4 pt);
\fill (4,0) circle (4 pt);
\fill (4,2) circle (4 pt);
\fill (4,4) circle (4 pt);
\fill (10,0) circle (4 pt);
\fill (10,2) circle (4 pt);
\fill (10,4) circle (4 pt);
\fill (10,6) circle (4 pt);
\fill (12,0) circle (4 pt);
\fill (12,1) circle (4 pt);
\fill (12,3) circle (4 pt);
\fill (12,4) circle (4 pt);
\fill (12,5) circle (4 pt);
\fill (12,7) circle (4 pt);
\fill (14,0) circle (4 pt);
\fill (14,2) circle (4 pt);
\fill (14,4) circle (4 pt);
\fill (14,6) circle (4 pt);
\fill (20,1) circle (4 pt);
\fill (20,3) circle (4 pt);
\fill (20,5) circle (4 pt);
\fill (20,7) circle (4 pt);
\draw(2,7)--(0,6)--(0,4)--(0,2)--(0,0);
\draw(2,7)--(4,6)--(4,4)--(4,2);
\draw(0,6)--(4,4)--(0,2)--(2,0);
\draw(4,6)--(0,4)--(4,2)--(0,0);
\draw[rounded corners] (0,6) -- (.5,5) -- (2,4);
\draw[rounded corners] (0,6) -- (1.5,5) -- (2,4);
\draw[rounded corners] (4,6) -- (2.5,5) -- (2,4);
\draw[rounded corners] (4,6) -- (3.5,5) -- (2,4);
\draw(2,4)--(0,2);
\draw(2,4)--(4,2)--(2,0);
\draw[rounded corners] (0,2) -- (2.5,1) -- (4,0);
\draw[rounded corners] (0,2) -- (1.5,1) -- (4,0);
\draw[rounded corners] (4,2) -- (4.25,1) -- (4,0);
\draw[rounded corners] (4,2) -- (3.75,1) -- (4,0);
\draw(12,7)--(10,6)--(12,5)--(10,4)--(12,3)--(10,2)--(12,1)--(10,0);
\draw(12,7)--(14,6)--(12,5)--(14,4)--(12,3)--(14,2)--(12,1)--(12,0);
\draw[rounded corners](12,5)--(11.75,4.5)--(12,4);
\draw[rounded corners](12,5)--(12.25,4.5)--(12,4)--(12,3);
\draw[rounded corners](12,1)--(12.75,.25)--(14,0);
\draw[rounded corners](12,1)--(13.25,.75)--(14,0);
\draw[rounded corners](20,7)--(19.75,6)--(20,5)--(19.5,4)--(20,3)--(19.75,2)--(20,1);
\draw[rounded corners](20,7)--(20.25,6)--(20,5)--(20.5,4)--(20,3)--(20.25,2)--(20,1);
\draw[rounded corners](20,5)--(20.2,4)--(20,3);
\draw[rounded corners](20,5)--(19.8,4)--(20,3);
\node[above] at(2,7){$B$};
\node[above] at (12,7){$B'$};
\node[above] at(20,7){Odometer};
\end{tikzpicture}
\end{center}
\caption{$B'$ telescopes to $B$ on odd levels, and to the odometer on even levels. Give $B$ any uniform ordering and the diagrams are order equivalent.}
\label{ID}
\end{figure}

We are seeking conditions that will allow one to determine when BV systems are expansive or not, conjugate to odometers or not, and of finite topological rank or not. As mentioned above, Downarowicz and Maass \cite{DownarowiczMaass2008} (see also \cite{Hoynes2014}) showed that every properly ordered minimal BV system of finite {topological} rank is either expansive or else conjugate to an odometer. Gjerde and Johansen \cite{GjerdeJohansen2000} give an example of a simple properly ordered BV system that is not expansive, yet not conjugate to an odometer, thus necessarily of infinite {topological} rank. We note in the following Proposition that the key properties of this example are enjoyed by a wide class of systems, which we call ``standard nonexpansive'', and which therefore are nonexpansive and, by essentially the same argument as in \cite{GjerdeJohansen2000}, not of finite {topological} rank, in particular not conjugate to odometers.

Recall that given $N \in \mathbb{N}$, every path in an ordered BV system $B=(\mathscr V, \mathscr E, \xi)$ has a $N$-coding of its orbit according to the partition $\mathscr P_N$ into cylinder sets to level $N$.
For each vertex $v\in \mathscr{V}$, the edges with range $v$ are assigned labels indicating their positions in the ordering determined by $\xi$. So any path $\gamma$ has an infinite \emph{sequence of edge labels} associated with it. Specifically, the $i$th entry in the sequence is the label assigned to the $i$th edge of $\gamma$.  We say that the BV system is \emph{standard nonexpansive} if for every natural number $N$ there exist distinct paths $\gamma _N$ and $\gamma '_N$ that are assigned the same sequence of edge labels, agree {from the root to} some level $K_N > N$, and at any level the vertices met by the two paths have the same basic block in the {$N$-coding}. In this case, $\gamma _N$ and $\gamma '_N$ have the same $N$-coding. Standard nonexpansive implies expansive.

\begin{prop}\label{SNE}
Every standard nonexpansive simple properly ordered BV system has infinite rank.
\end{prop}
\begin{proof}
By \cite{HermanPutnamSkau1992}, for any BV system $B$ that is conjugate to the odometer there is a BV system $B'$ that telescopes on odd levels to a telescoping of $B$ and on even levels to a telescoping of the odometer. There is then a telescoping of $B$ in which every vertex is uniformly ordered. This, however, cannot occur when there are paths with the same sequence of edge labels. So if $B$ is standard nonexpansive, then it not conjugate to the odometer. It then follows, by \cite{DownarowiczMaass2008}, that $B$ does not have finite {topological} rank.
\end{proof}

We consider now what might happen if a diagram is ordered randomly, or if a diagram is somehow chosen at random.

Suppose that we are given a simple diagram $B=(\mathscr V, \mathscr E)$ as above. Denote the set of orderings of the edges entering a vertex $v \in \mathscr V$ by $\mathscr O_v$. The set of all orderings of $B$ is $\mathscr O= \prod_{v \in \mathscr V} \mathscr O_v$. We consider probability measures $\mu$ on $\mathscr O$, for example the uniform and independent measure $\mu_0$ which is the product over all $v$ of the uniform measures on the discrete spaces $\mathscr O_v$. For each $n=1,2,\dots$ let $U_n$ denote the set of orderings of $B$ which are uniformly ordered at level $n$, and let $U$ denote the set of orderings which have infinitely many uniformly ordered levels. By the preceding theorem, if $\mu (U)=1$ then an ordering selected randomly according to $\mu$ will with probability 1 produce an adic system topologically conjugate to an odometer.

\begin{ex}\label{ex:odometerprob}
 Suppose that $B$ is a diagram such that for every $n\geq 1$ there is an $r_n$ such that the number of edges between every pair of vertices $v_{n-1}\in \mathscr{V}_{n-1}$ and $v_{n}\in\mathscr{V}_{n}$ is $r_n$. (Then $B$ is simple.) Write $|\mathscr V_n|=V_n$. Then there are $r_n!$ orderings of the edges into each vertex at level $n$ and thus $(r_n!)^{V_n}$ ways to order the unordered shape at level $n$. There are $r_n!$ uniform orderings at level $n$, since we must have the same ordering at each vertex on a fixed level. By the Borel-Cantelli Lemma, if
\begin{equation}
\sum_{n=1}^\infty \frac{r_n!}{(r_n!)^{V_n}} = \infty,
\end{equation}
then $\mu_0(U)=1$, and so with probability 1 the system is topologically conjugate to an odometer. For example, if the $r_n$ and $V_n$ are bounded, and $B$ is as described, then selecting orderings uniformly and independently yields
an odometer with probability 1. This is consistent with Corollary 5.3 of \cite{BezuglyiKwiatkowskiYassawi2014}, according to which a diagram with bounded number of vertices and edges per level will have a ``perfect'' ordering (an ordering which admits a Vershik map as a homeomorphism) with probability 1.
\end{ex}

\begin{ex}\label{ex:odometer2}
We can also consider assembling a random ordered diagram from a sequence of shapes, that fit together, generated by a stochastic process.
Generalizing the definition above, define an {\em ordered shape} $S=(C,D,E,\xi)$ to be a bipartite directed graph together with a partial order on its set of edges such that (1) the vertex set is the disjoint union of two finite sets $C$ and $D$; (2) for each $e \in E, s(e) \in C$ and $r(e) \in D$; (3) for each
$b \in D, \xi$ is a total order on $r^{-1}b$; (4) for $b_1,b_2 \in D$ with $b_1 \neq b_2$, the edges in $r^{-1}b_1$ and $r^{-1}b_2$ are not comparable.

Let $\mathscr S$ be the set of sequences of ordered shapes $S_n=( C_(S_n),D(S_n),E(S_n),\xi(S_n)), n=1,2,\dots$, such that $ C(S_{n+1})=D(S_n)$ for all $n$, so that the shapes can ``fit together". If $\mu$ is a probability measure on $\mathscr S$ that assigns full measure to the set of sequences which have infinitely many uniformly ordered shapes, then constructing an adic system at random with respect to $\mu$ produces an odometer with probability 1. For example, if $\mu$ is an ergodic shift-invariant probability measure on $\mathscr S$ such that $\mu\{S=(S_n) \in \mathscr S: S_1 \text{ is uniformly ordered}\}>0$, then by the Poincar\'{e} Recurrence Theorem $S$ will have infinitely many uniformly ordered shapes and hence will describe an odometer.
\end{ex}

\section{The Union Over All Orderings}

 For a subshift $(\Omega,\sigma)$, denote by $\mathscr L (\Omega)$ the {\em language} of $\Omega$, that is, the set of finite blocks seen in sequences in $\Omega$. Define for each ordering $\xi$ of the Pascal graph $P$ the subshift $\Sigma_\xi \subset \{a,b\}^\Z$ to consist of all doubly infinite sequences, each of whose subblocks can be found in the 1-coding of some $\gamma \in X$ (the set of infinite paths on the Pascal graph) for the ordering $\xi$, and define the ``big subshift" $(\Sigma,\sigma)$ by $\mathscr L (\Sigma) = \cup_\xi \mathscr L (\Sigma_\xi)$. Equivalently, we can define the big subshift $\Sigma$ to be the closure of the union over all orderings $\xi$ of $X_\xi$:
\begin{equation}
\Sigma= \overline{\bigcup_\xi X_\xi} .
\end{equation} What is the asymptotic complexity of $\Sigma$? What are other dynamical properties of $(\Sigma, \sigma)$? Is this system topologically transitive (i.e., does it have a dense orbit)? Are there periodic points?

Recall the basic blocks $B_\xi (x,y)$ defined above.
Basic blocks are determined by the orderings on the edges above them. See Figure \ref{fig:determineblocks}.
 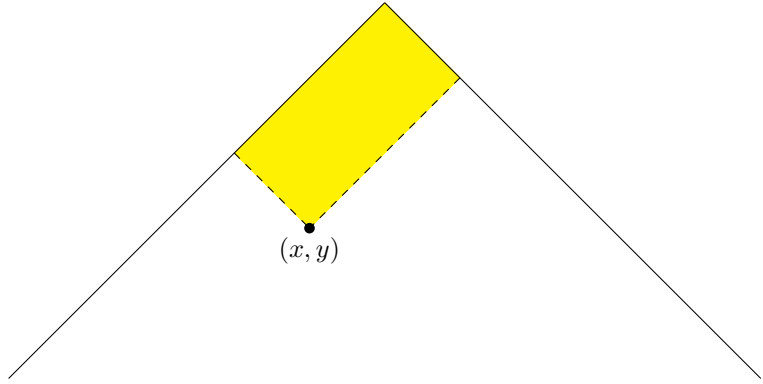
\begin{figure}[h]
\begin{center}
\begin{tikzpicture}[scale=1]
\draw[fill=yellow, yellow] (3,3)--(4,2)--(6,4)--(5,5);
\draw (5,5) -- (0,0);
\draw (5,5) -- (10,0);
\fill (4,2) circle (2 pt);
\node[below] at (4,2) {$(x,y)$};
\draw[dashed](6,4)-- (4,2);
\draw[dashed](3,3)-- (4,2);
\end{tikzpicture}
\end{center}
\caption{Each choice of ordering of the incoming edges to the vertices in the shaded region will determine a $B_{\xi}(x,y)$.}
\label{fig:determineblocks}
\end{figure}
However, a block may be generated by different orderings. For example, there exist two distinct orderings, $\xi_1$ and $\xi_2$ such that $B_{\xi_1}(2,1)=B_{\xi_2}(2,1)=aba$. This can continue throughout the diagrams. See Figure \ref{fig:differentblocks}.

\begin{figure}
\begin{tikzpicture}[scale=.5]
\foreach \n in {0,...,6}{ 
\foreach \k in {0,...,\n}{ 
\path(2*\k-\n,-\n) coordinate (V\n\k); 
\fill (V\n\k) circle (3 pt);
}}
\draw(0,0) --++(6,-6);
\draw(0,0)--++(-6,-6);
\draw(1,-1) --++(-1,-1);\draw[dashed](0,-2) to (-1,-3);\draw(-1,-3) to (-2,-4);\draw[dashed](-2,-4) to (-3,-5);
\draw[dashed](-1,-1) -- ++(2,-2);
\draw(2,-2) --++(-3,-3);\draw[dashed](-1,-5) to (-2,-6);
\draw(-2,-2)--++(1,-1);\draw[dashed](-1,-3)--++(1,-1);
\draw[dashed](-3,-3)--++(2,-2);
\draw(-4,-4)--++(2,-2);
\node[below]at (-2,-6){$(4,2)$};
\node[left] at (-3,-2.5){$\xi_1$};
\end{tikzpicture}
\hskip .25in
\begin{tikzpicture}[scale=.5]
\foreach \n in {0,...,6}{ 
\foreach \k in {0,...,\n}{ 
\path(2*\k-\n,-\n) coordinate (V\n\k); 
\fill (V\n\k) circle (3 pt);
}}
\draw(0,0) --++(6,-6);
\draw(0,0)--++(-6,-6);
\draw(-1,-1)--++(2,-2);
\draw(-2,-2)--++(1,-1);\draw[dashed](-1,-3)--+(1,-1);
\draw[dashed](-3,-3)--+(1,-1);
\draw(-2,-4)--+(1,-1);
\draw[dashed](-4,-4)--+(2,-2);
\draw[dashed](1,-1)--+(-2,-2);\draw(-1,-3)--+(-2,-2);
\draw[dashed](2,-2)--+(-1,-1);\draw(1,-3)--+(-1,-1);\draw[dashed](0,-4)--+(-1,-1);\draw(-1,-5)--+(-1,-1);
\node[below]at(-2,-6){$(4,2)$};
\node[right] at (3,-2.5){$\xi_2$};
\end{tikzpicture}
\caption{Dashed lines represent minimal edges. Both orderings $\xi_1$ and $\xi_2$ give rise to $B(4,2)=a^2ba^2ba^2b^2a^2ba^2$.}
\label{fig:differentblocks}
\end{figure}
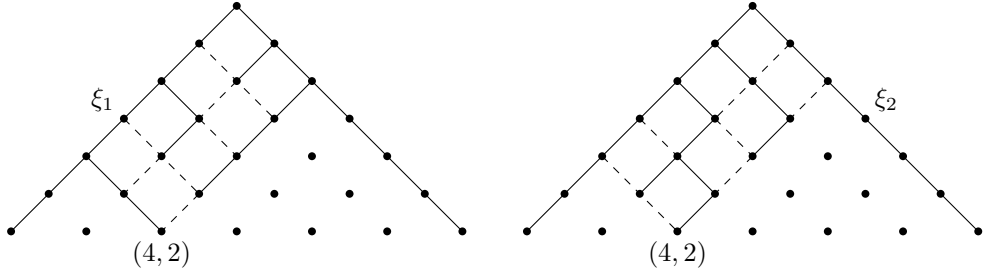

Assuming a fixed ordering $\xi$, the basic blocks $B(x,1)$, for $x \geq 1$, will be used frequently in the following and we will give them the special name $C_x$. In a similar fashion, for $y \geq 1$ define the block $D_y=B(1,y)$.

In the following we will consider the set of diagrams for which
\[\xi(x,1)=\xi(2,y)=0.\]In other words, the ordering on the edges coming into vertices $(x,1)$ and $(1,y)$ is left to right for all values of $x\geq 1$ and $y\geq 1$. This means that $C_i=a^ib$ and $D_j=ab^j$ for $i,j\geq 1$. In the following we will show that for such orderings basic blocks have unique factorizations.

\begin{lem}\label{lem:decomp} Given $(x,y)$ with $x \geq 1$ and $y \geq 1$, there is a one-to-one correspondence between the set of possible basic blocks \[\{B_{\xi}(x,y):\xi(i,1)=\xi(1,j)=0\text{ for all }i,j\geq 1\}\] and the set of orderings of edges into each vertex $(u,v)$ with $u \leq x$ and $v \leq y$: $B_\xi (x,y)=B_\eta (x,y)$ if and only if $\xi(u,v)=\eta(u,v)$ for all $(u,v)$ with $u \leq x$ and $v \leq y$. \end{lem}

\begin{proof}
The reverse direction is clear from the definition of $B_{\xi}(x,y)$. We will therefore consider a $B_{\xi}(x,y)$ and show that it gives rise to a unique $\xi$, given the appropriate restrictions.

Assume $x,y>1$ and we are given a basic block $B_\xi(x,y)$. Starting from the left in $B_\xi(x,y)$, find the first occurrence of a block {of the form, for some $i_0,j_0 \geq 1$,
$a^{i_0}b^{j_0}a$, or $a^{i_0}b^{j_0}$ if no $a^{i_0}b^{j_0}a$ is to be found.} Because of the restriction on $\xi$, we know that it is either the case that $i_0=1$ and $j_0\neq 1$ or $i_0\neq 1$ and $j_0=1$. In the first case, $a^{i_0}b^{j_0}$ corresponds to the block $D_{j_0}$, and in the second to $C_{i_0}$. Continue {left to right} in this fashion to decompose $B_\xi(x,y)$ completely into a sequence of $C_i$'s and $D_j$'s. Note that there is a unique path between {$(x,1)$} and $(x,y)$, and therefore $C_x$ appears only once in $B_\xi(x,y)$. Furthermore, if $i>x$ then there is no path from $(i,1)$ to $(x,y)$, and therefore $C_i$ cannot be a subblock of $B(x,y)$. A similar argument gives the result for the $D_y$. So $C_x$ and $D_y$ each appear exactly once, and if $C_i$ (or $D_j$) appears in $B_\xi(x,y)$, then $i{\leq}x$ (or $j\leq y$).

Now, $B_\xi(x,y)$ decomposes into the concatenation of $B_\xi(x-1,y)$ and $B_\xi(x,y-1)$. Note that the unique $C_x$ in $B_\xi(x,y)$ is a subblock of the $B_\xi(x,y-1)$ portion of $B_{\xi}(x,y)$. Likewise, $D_y$ is a aubblock of the $B_\xi(x-1,y)$ portion of $B_\xi(x,y)$. Therefore, if $C_x$ appears before $D_y$, then $B_\xi(x,y)=B_\xi(x,y-1)B_\xi(x-1,y)$, and vice versa. This determines the order that $\xi$ gives to the edges with range $(x,y)$. We also know the lengths of each of $B(x,y-1)$ and $B(x-1,y)$. Once we know which one comes first, we know exactly where the juxtaposition between the two occurs. We repeat this process on each word above $B(x,y)$ to
conclude that no ordering besides $\xi$ could have produced the block $B_\xi(x,y)$. \end{proof}

We work through an example.
\begin{ex} Consider the following block:
$$ab^3ab^2a^2ba^3bab^2a^2ba^3bab^2a^2ba^4b$$
Because we know the structure of the $C_i$ and $D_j$, we see that the first set of $a$'s and $b$'s is $ab^3$. This must be $D_3$. Then the next set of $a$'s and $b$'s is $ab^2$. This must be $D_2$. We continue in this manner to see that the above block decomposes into $D_3D_2C_2C_3D_2C_2C_3D_2C_2C_4$. By the above lemma, this must be block $B(4,3)$. Because the $D_3$ appears before the $C_4$, $B(4,3)$ must decompose into $B(3,3)B(4,2)$, which implies there was a right to left ordering on the edges with range $(4,3)$. We need only to know that $|B(3,3)|=20$ to see that $B(3,3)=D_3D_2C_2C_3D_2C_2$ and $B(4,2)=C_3D_2C_2C_4$. We continue in this manner until we have placed an order on the incoming edges of all 6 unordered vertices that are above $(4,3)$. See Figure \ref{fig:Decomposition}.\end{ex}

\begin{figure}[h]
\begin{tikzpicture}[scale=.5]
\foreach \n in {0,...,7}{ 
\foreach \k in {0,...,\n}{ 
\path(2*\k-\n,-\n) coordinate (V\n\k); 
\fill (V\n\k) circle (3 pt);
}}
\draw(0,0) --++(7,-7);
\draw(0,0)--++(-7,-7);
\draw[dashed](-1,-1)--+(3,-3);
\draw[dashed](-2,-2)--+(1,-1);\draw(-1,-3)--+(2,-2);
\draw[dashed](-3,-3)--+(2,-2);\draw(-1,-5)--+(1,-1);
\draw[dashed](-4,-4)--+(2,-2);\draw(-2,-6)--+(1,-1);
\draw(1,-1)--+(-4,-4);
\draw(2,-2)--+(-1,-1);\draw[dashed](1,-3)--+(-1,-1);\draw(0,-4)--+(-2,-2);
\draw(3,-3)--+(-1,-1);\draw[dashed](2,-4)--+(-3,-3);
\end{tikzpicture}
\caption{The ordering of edges determined by $B(4,3)=ab^3ab^2a^2ba^3bab^2a^2ba^3bab^2a^2ba^4b$}
\label{fig:Decomposition}
\end{figure}
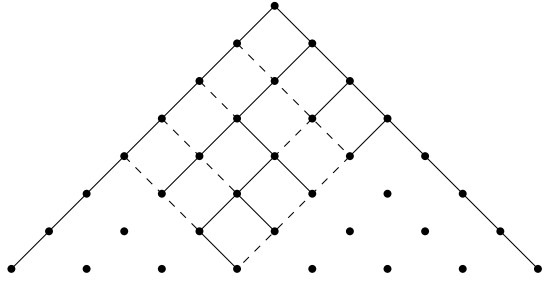

\begin{prop} Given left to right ordering on edges terminating at vertices $(u,1)$ and $(1,v)$ for all $u,v\geq 1$, then for all $x,y\geq 1$ there are $2^{(x-1)(y-1)}$ distinct blocks $B_\xi(x,y)$ as we vary the orderings $\xi$.
\label{Prop:DistinctB}\end{prop}
\begin{proof} By the Lemma, each ordering $\xi$ of the edges entering vertices $(u,v)$ above $(x,y)$ ($1 \leq u \leq x, 1 \leq v \leq y$) gives rise to a distinct $B_\xi(x,y)$. There are $(x-1)(y-1)$ such vertices $(u,v)$ above $(x,y)$ and two choices of ordering for each vertex, yielding $2^{(x-1)(y-1)}$ distinct orderings and hence that many blocks $B_\xi(x,y)$. \end{proof}

\begin{prop} For the big subshift, the complexity function $p(n)$ is superpolynomial: for each $k=2,3,\dots$ there are at least $2^{k-1}$ blocks of length $(k+1)(k+2)/2$ in the big subshift.
\end{prop}
\begin{proof}
Let $k=2,3,\dots$ and $n=(k+1)(k+2)/2$. Note that the binomial coefficient $ C(k+2,2)=n$. Then for any ordering $\xi$, $B_{\xi}(k,2)$ has length $n$ since the length of $B(k,2)$ is the binomial coefficient $ C(k+2,2)$.

 Then the coarse lower bound above is given by just looking at the $2^{k-1}$ unique ways to form $B_\xi(k,2)$ and applying Proposition \ref{Prop:DistinctB}.
\end{proof}

We will show next that the big subshift $\Sigma$ is not topologically transitive and that the ``small subshift" whose language is $\cap_\xi \mathscr L (\Sigma_\xi)$ consists of only the two fixed points and the orbits of the points $a^\infty . b a^\infty$ and $b^\infty . a b^\infty$.

\begin{lem}\label{lem:ba^9}
(1) There are orderings $\xi$ and $\xi'$ of the Pascal graph such that $B_\xi(3,3)=a(ab)^9b$ and {$B_{\xi'}(3,3)=b(ba)^9a$}.
(2) For every ordering $\xi$ and every $x,y \geq 0$, the basic block $B_\xi (x,y)$ does not contain both $(ab)^9$ and $(ba)^9$.
\end{lem}
\begin{proof}

To prove (1), we proceed to define the orderings $\xi$ and $\xi'$ called for in the Lemma. Define $\xi(1,1)=\xi'(1,1)=0$, so that $B_\xi(1,1)=B_{\xi'}(1,1)=ab$.
Further, choose $\xi$ and $\xi'$ so that $B_\xi(2,1)=B_{\xi'}(2,1)=aba$ and $B_\xi(1,2)=B_{\xi'}(1,2)=bab$. {Continuing, define $\xi$ and $\xi'$ so that $B_\xi(3,1)=a^2ba$, $B_\xi(2,2)=(ba)^3$, and $B_\xi(1,3)=bab^2$, while $B_{\xi'}(3,1)=aba^2$, $B_{\xi'}(2,2)=(ab)^3$, and $B_{\xi'}(1,3)=b^2ab$.}
Then define $\xi(3,2)=\xi(2,3)=\xi(3,3)=0$ (so that the ordering by $\xi$ is ``left-right'' at these vertices) and $\xi'(3,2)=\xi'(2,3)=\xi'(3,3)=1$ (so that the ordering by $\xi'$ is ``right-left'' at these vertices).
We then have $B_\xi(3,3)=a(ab)^9b$ and $B_{\xi'}(3,3)=b(ba)^9a$. See Figure \ref{fig:ab9ba9}.
\begin{figure}[h]
\begin{tikzpicture}[scale=.5]
\foreach \n in {0,...,6}{ 
\foreach \k in {0,...,\n}{ 
\path(2*\k-\n,-\n) coordinate (V\n\k); 
\fill (V\n\k) circle (3 pt);
}}
\draw(0,0) --++(6,-6);\draw(0,0)--++(-6,-6);
\draw(-1,-1)--+(-1,-1);\draw[dashed](-1,-1)--+(1,-1);
\draw(1,-1)--+(-1,-1);\draw(1,-1)--+(1,-1);
\draw(-2,-2)--+(-1,-1);\draw(-2,-2)--+(1,-1);
\draw(0,-2)[dashed]--+(-1,-1);\draw(0,-2)--+(1,-1);
\draw(2,-2)[dashed]--+(-1,-1);\draw(2,-2)--+(1,-1);
\draw[dashed](-3,-3)--+(1,-1);
\draw(-1,-3)--+(-1,-1);\draw(-1,-3)--+(1,-1);
\draw(1,-3)[dashed]--+(-1,-1);\draw(1,-3)[dashed]--+(1,-1);
\draw(3,-3)--+(-1,-1);
\draw[dashed](-2,-4)--+(1,-1);
\draw(0,-4)--+(-1,-1);\draw[dashed](0,-4)--+(1,-1);
\draw(2,-4)--+(-1,-1);
\draw(-1,-5)[dashed]--+(1,-1);
\draw(1,-5)--+(-1,-1);
\node[left] at (-3,-2.5){$\xi$};
\end{tikzpicture}
\begin{tikzpicture}[scale=.5]
\foreach \n in {0,...,6}{ 
\foreach \k in {0,...,\n}{ 
\path(2*\k-\n,-\n) coordinate (V\n\k); 
\fill (V\n\k) circle (3 pt);
}}
\draw(0,0) --++(6,-6);\draw(0,0)--++(-6,-6);
\draw(-1,-1)--+(-1,-1);\draw[dashed](-1,-1)--+(1,-1);
\draw(1,-1)--+(-1,-1);\draw(1,-1)--+(1,-1);
\draw(-2,-2)--+(-1,-1);\draw(-2,-2)--+(1,-1);
\draw(0,-2)[dashed]--+(-1,-1);\draw(0,-2)--+(1,-1);
\draw(2,-2)[dashed]--+(-1,-1);\draw(2,-2)--+(1,-1);
\draw(-3,-3)--+(1,-1);
\draw[dashed](-1,-3)--+(-1,-1);\draw[dashed](-1,-3)--+(1,-1);
\draw(1,-3)--+(-1,-1);\draw(1,-3)--+(1,-1);
\draw(3,-3)[dashed]--+(-1,-1);
\draw(-2,-4)--+(1,-1);
\draw[dashed](0,-4)--+(-1,-1);\draw(0,-4)--+(1,-1);
\draw[dashed](2,-4)--+(-1,-1);
\draw(-1,-5)--+(1,-1);
\draw(1,-5)[dashed]--+(-1,-1);
\node[right] at (3,-2.5){$\xi'$};
\end{tikzpicture}
\caption{{Dashed lines represent minimal edges. $B_{\xi}(3,3)=a(ab)^9b$ and $B_{\xi'}=b(ba)^9a$.}}
\label{fig:ab9ba9}
\end{figure}

To prove (2), we want to focus on
 strings of alternating $a$'s and $b$'s that appear in the language of $\Sigma_\xi$ for some ordering $\xi$ of the Pascal graph and hence in a basic block $B_\xi(x,y)$ for some $x,y \geq 0$.
 So we provide a systematic way to ignore subblocks of the $B_\xi(x,y)$ that cannot be used, when $B_\xi(x,y)$ is concatenated with $B_\xi(x-1,y+1)$ or $B_\xi(x+1,y-1)$, to create a longer string $(ab)^j$ than has been created previously.
Specifically, for each $x,y \geq 1$ we let $Q_\xi ^L(x,y)$ (respectively $Q_\xi ^R(x,y)$) be the longest string in which $a$ and $b$ alternate that is a prefix of $B_\xi(x,y)$ (respectively suffix of $B_\xi(x,y)$).
For example, if $B_\xi(x,y)=ababa^2b$, then $Q_\xi^ L(x,y)=ababa$ and $Q_\xi ^R(x,y)=ab$.
We add the new symbol $*$ to our alphabet and define the ``condensed form'' of $B_\xi(x,y)$ to be $Q_\xi(x,y)=Q_\xi ^L(x,y) * Q_\xi ^R(x,y)$. (If the block $B_{\xi}(x,y)$ is made up entirely of alternating $a$'s and $b$'s there is no $*$ in $Q_\xi(x,y)$.)

We want to show that {if $\xi$ is an ordering such that $(ab)^9 \in \mathscr L(\Sigma_\xi)$, then  $(ba)^9 \notin \mathscr L(\Sigma_\xi$), and vice versa}; that is, no $B_\xi(x,y)$ can ever contain both $(ab)^9$ and $(ba)^9$. We begin with two observations.
The first is that for $x\geq 3$, $Q_\xi^L(x,1)$ cannot have length greater than 3. (Specifically, $Q_\xi^L(x,1)$ can be $aba$ but cannot be any longer.) The same is true for $Q_\xi^R(x,1)$, and similarly for $Q_\xi^L(1,y)$ and $Q_\xi^R(1,y)$ for $y\geq 3$.
The second observation is that if a basic block $B_1$ has condensed form $C_1^L*C_1^R$ and a second basic block $B_2$ has condensed form $C_2^L*C_2^R$, then the concatenation $B_1B_2$ has condensed form $C_1^L*C_2^R$.
Therefore, in order to determine the highest power of $(ab)^j$ that can be formed by a concatenation of any two basic blocks one needs to look only far enough down to a level $n$ for which the condensed forms of all the basic blocks $B_{\xi}(x,y)$ with $x+y=n$ and $x,y\geq 1$ contain the symbol *.
We can then determine the highest power of $(ab)^j$ in $\mathscr L (\Sigma_\xi)$ by considering all possible concatenations $Q_\xi^R(i,j)Q_\xi^L(k,l)$ with $i+j=k+l=n$ and $i,j,k,l\geq 1$.

Without loss of generality we assume that $B_\xi(1,1)=ab$. Then depending on the choice among the four partial orders on the edges with range in $\mathscr{V}_3$ there are the following possibilities for the basic blocks at level $3$: $B_\xi(3,0)=a$, $B_\xi(2,1)=a^2b$ or $aba$, $B_\xi(1,2)=ab^2$ or $bab$ and $B_\xi(0,3)=b$.
 It is {straightforward to check that for three of the four possible orderings all the condensed blocks $Q_\xi(x,y)$ with $x+y=4$ and $x,y\geq 1$ necessarily} contain the symbol $*$.
 Since every basic block at level 4 has length no more than 6, {for these three orderings} it is impossible that $(ab)^9$ or $(ba)^9$ is contained within any of these blocks or, {since they all contain a $*$,} any of their concatenations.
 Hence the only way that $(ab)^9$ or $(ba)^9$ could be formed is in the case {of the fourth ordering at level 3}, when $B_\xi(2,1)=aba$ and $B_\xi(1,2)=bab$.
 In this case the possibilities are $Q_\xi(3,1)=a*aba$ or $aba*a$, $Q_\xi(2,2)=(ab)^3$ or $(ba)^3$, and $Q_\xi(1,3)=bab*b$ or $b*bab$. For each of the 8 possible orderings {at the vertices $(x,y)$ with $x+y=5$ and $x,y\geq 1$ the condensed blocks $Q_\xi(x,y)$ all necessarily contain a $*$.} Therefore 5 is the last level at which we need to look in order to  determine which orderings give rise to either $(ab)^9$ or $(ba)^9$.

  The only two blocks on level 5 that could concatenate to give a word of length $18$ are $B_\xi(3,2)$ and $B_\xi(2,3)$, both of length 10.
 Further, since each of $Q_\xi(3,2)$ and $Q_\xi(2,3)$ contains a $*$, the maximum length of $Q_\xi^L(3,2), Q_\xi^R(3,2), Q_\xi^L(2,3), Q_\xi^R(2,3)$ is 9.
 In order to form either $(ab)^9$ or $(ba)^9$, we would need both $Q_\xi^L(3,2)$ and $Q_\xi^R(2,3)$ to be of length 9, or both $Q_\xi^R(3,2)$ and $Q_\xi^L(2,3)$ to be of length 9.
 {In the situation to which we have reduced, there is exactly one ordering down to level 5 for which both $Q_\xi^L(3,2)$ and $Q_\xi^R(2,3)$ have length 9; for this ordering $Q_\xi(3,2)=(ab)^4a*a$ and $Q_\xi(2,3)=b*b(ab)^4$.
 In this case it is possible to concatenate $B_\xi(3,2)$ and $B_\xi(2,3)$ to form the word $(ba)^9$.
 However, because these two blocks contain the longest possible strings of alternating $a$'s and $b$'s {down to level 5}, there is no way, given this ordering, ever to create the word $(ab)^9$.}
 Similarly, there is exactly one case for which $Q_\xi^R(3,2)$ and $Q_\xi^L(2,3)$ are of length 9, namely when {$Q_\xi^R(3,2)=a*a(ba)^4$ and $Q_\xi^L(2,3)=(ba)^4b*b$.}
 The string $(ab)^9$ can be made with the ordering corresponding to this case, but $(ba)^9$ cannot.
 Therefore, any ordering which gives rise to the block $(ab)^9$ cannot also give rise to $(ba)^9$, and vice versa.
\end{proof}

\begin{thm} There is no dense orbit in $\Sigma$. \end{thm}
\begin{proof}
By the Lemma, both of the blocks $(ab)^9$ and $(ba)^9$ are in $\mathscr L(\Sigma)=\cup_\xi \mathscr L(\Sigma_\xi)$. But no sequence in $\Sigma$ can contain both of these blocks, since any block containing both of them would have to be a subblock of $B_\xi(x,y)$ for some $\xi, x$, and $y$.
 \end{proof}

\begin{prop}\label{prop:periodicpoints}
The only periodic points in the big subshift $(\Sigma,\sigma)$ are the two fixed points $a^\infty$ and $b^\infty$.
\end{prop}
\begin{proof}
Suppose that $\omega$ is a periodic sequence. If $\omega$ is not one of the two fixed points, then the runs of $a$ and $b$ in $\omega$ are of bounded length: there is $r < \infty$ such that if $a^j$ or $b^j$ appear in $\omega$, then $j < r$.
For each ordering $\xi$ the two basic blocks $B_\xi(2r,1)$ and $B_\xi(1,2r)$ each contain strings of length $r$ of $a$ or $b$, respectively.
If $n > 4r$, then every basic block $B_\xi(x,y)$ with $x+y=n$ and $x,y \geq 1$, which is a concatenation of basic blocks at level $2r+1$, contains one of the basic blocks $B_\xi(1,2r)$ or $B_\xi(2r,1)$ from level $2r+1$.
Denote by $M$ the maximum of the lengths of basic blocks at level $4r$.
Every block in the language $\mathscr L (\Sigma)$ of length greater than $3M$ must contain either a basic block of length greater than $1$ from level $4r$, and hence $a^j$ or $b^j$ for some $j\geq r$, or one of the strings $a^M$ or $b^M$. The later cannot occur since $M>r$. Thus no block in $\omega$ of length greater than $3M$ can be in $\mathscr L (\Sigma)$, and $\omega$ cannot be in $\Sigma$ unless it is one of the two fixed points.
\end{proof}

\begin{lem}\label{lem:fixedpts}
The small subshift $\cap_\xi \Sigma_\xi$ contains the points $a^\infty, b^\infty, a^\infty.ba^\infty$, and $b^\infty .ab^\infty$.
\end{lem}
\begin{proof}
If $\xi(x,1)$ is not eventually constant, then for every $r>0$ there is an $x$ such that $B_\xi(x,1)=a^iba^j$ and $i,j>r$. Then $a^\infty.ba^\infty$ and $a^\infty$ are in $\Sigma_\xi$.

Suppose then that $\xi(x,1)$ is eventually constant, say $\xi(x,1)=0$ for all $x>K$. Then $B_\xi(x,1)=a^{x-m}ba^m$ for some $m\leq x$ and all $x>K$.
{The block $B_\xi(x,2)$ is a concatenation of the blocks $B_\xi(x-3,2)$, $B_\xi(x-2,1)$, $B_\xi(x-1,1)$ and $B_\xi(x,1)$ in some order, each with multiplicity one.} By the the pigeonhole principle it is necessary that either $B_\xi(x-2,1)$ is concatenated with $B_\xi(x-1,1)$, or $B_\xi(x-2,1)$ is concatenated with $B_\xi(x,1)$, or $B_\xi(x-1,1)$ is concatenated with $B_\xi(x,1)$.  It follows that for $x>K+3$ the block $a^{x-m-2}ba^{x-2}$ necessarily appears in $B_\xi(x,2)$. Letting $x$ tend to infinity shows that $a^\infty.ba^\infty$ and $a^\infty$ are in $\Sigma_\xi$.

That $b^\infty.ab^\infty, b^\infty \in \Sigma_\xi$ follows by symmetry.
\end{proof}

\begin{thm}\label{thm:smallsubshift}
The small subshift $\cap_\xi \Sigma_\xi$ consists of only the two fixed points and the orbits of the points $a^\infty . b
a^\infty$ and $b^\infty . a b^\infty$.
\end{thm}
\begin{proof}
We fix the orderings $\xi$ and $\xi'$ used in Lemma \ref{lem:ba^9}:
\begin{equation}
\begin{gathered}
\xi(1,1)=\xi'(1,1)=0, \xi(1,2)=\xi'(1,2)=\xi(2,1)=\xi'(2,1)=1,\\
\xi(2,2)=1,\xi'(2,2)=0; \xi(1,3)=\xi(3,1)=0, \xi'(1,3)=\xi'(3,1)=1\\
\text{for all }{ (x,y)  \text{ with }} x+y \geq {5}, \, \xi(x,y)=0;\\
\text{for all } { (x,y)  \text{ with }} x+y \geq {5}, \, \xi'(x,y)=1.
\end{gathered}
\end{equation}

For notational ease, let $B_{\xi}(j,2)$ be denoted {by} $F_j$,  let $B_{\xi}(2,k)$ be denoted {by} $G_k$, and { let} $B_{\xi}(3,3)=a(ab)^9b=H$.  Likewise,  let $B_{\xi'}(j,2)$ be denoted {by} $F'_j$,  let $B_{\xi'}(2,k)$ be denoted {by} $G'_k$, and {let} $B_{\xi'}(3,3)=b(ba)^9a=H'$.
For all $ j, k \geq 4$,

\begin{equation}
\begin{aligned}
F_j=B_\xi(j,2)&=a^{j-1}ba^{j-1}ba^{j-2}ba^{j-3}b\dots a^4ba^3(ba)^4\\
G_k=B_\xi(2,k)&=(ba)^4b^3ab^4ab^5a\dots b^{k-1}ab^{k-1}
\end{aligned}
\end{equation}
{and}
\begin{equation}
\begin{aligned}
F'_j=B_{\xi'}(j,2)&=(ab)^4a^3ba^4ba^5b\dots b a^{j-1}ba^{j-1}\\
G'_k=B_{\xi'}(2,k)&=b^{k-1}ab^{k-1}ab^{k-2}ab^{k-3}\dots b^5ab^4ab^3(ab)^4.
\end{aligned}
\end{equation}
Note that if  $y\leq 1$, $B_{\xi}(x,y)$ and $B_{\xi'}(x,y)$ have at most one $b$; and if $x\leq 1$, then each of these blocks has at most one $a$.

We first determine ways that for $l>6$, $ba^lb$ may occur in $B_{\xi}(x,y)$, {for any $(x,y)$}.  {Because of the preceding observation,} it suffices to consider only  $x,y\geq 2$ .
If $x+y\geq 6$, $B_{\xi}(x,y)$ can be decomposed into a concatenation of $F_j$, $G_k$, and $H$ {with} $j,k \geq 4$ {as follows. Since $\xi(x,y)=0$ for all $x+y\geq 5$,  $B_{\xi}(x,y)=B_{\xi}(x,y-1)B_{\xi}(x-1,y)$.}
Continue factoring in this manner, {reducing $x$ and $y$,} until each block {in the factorization} is one of {$B_\xi(3,3)=H, B_\xi(j,2)=F_j$, or $B_\xi(2,k)=G_k$, for some $j,k \geq 4$}. The {concatenation {that we arrive at} must have the} following structure. $F_4$ can be followed {only} by $H$.
For $j>4$, $F_j$ can be followed {only} by $F_{j-1}$. $G_4$ can be preceded {only} by $H$. For $k>4$, $G_k$ can be preceded {only} by $G_{k-1}$. Note that for all $k,j\geq 4$ the concatenation $G_kF_j$ does occur in some $B_\xi(x,y)$, for example in $B_{\xi}(j+1,k+1)$.

{None of the blocks $H, F_j, G_k$ end with $a^2$; for all $l>2$, $k \geq 4$, $G_k$ and $H$ do not contain $a^l$.} So {if $l>6$ and $ba^lb$ appears in $B_\xi(x,y)$, then} the majority of the $a^l$ portion of $ba^lb$ must {appear within one of the factors $F_j$ of $B_\xi(x,y)$ for some $j \geq 4$.} We can then see that $ba^lb$ appears in one of {only} four possible ways:
 \begin{enumerate}
 \item across {a} concatenation $G_kF_{l+1}$ for {some $k \geq 4$,}
 \item across {a} concatenation $HF_{l+1}$,
 \item across {a} concatenation $F_{l+1}F_l$,
 \item entirely within {a} block $F_j$ for {some} $j>l$.
 \end{enumerate}
 In the first two cases, {the appearance of} $ba^lb$ {is within the block} $ba^lba^lba^{l-1}b$, while in the second two cases the appearance of $ba^lb$ is {within the block} $ba^lba^{l-1}b$.  {Therefore, $ba^lb$ appears in any sequence in $\Sigma_{\xi}$ only as a subblock of one of the two longer blocks $ba^lba^lba^{l-1}b$ or $ba^lba^{l-1}b$.}

We next determine ways that for $l>6$, $ba^lb$ may occur in $B_{\xi'}(x,y)$, for any $(x,y)$.   Again, it suffices to consider only $x,y\geq 2$.  If $x+y\geq 6$,  $B_{\xi'}(x,y)$  can  be decomposed into a concatenation of $F'_j$, $G'_k$ and $H'$ {with} $j, k \geq 4$ {as follows}.
{Since $\xi'(x,y)=1$ for all $x+y\geq 5$,  $B_{\xi'}(x,y)=B_{\xi'}(x-1,y)B_{\xi'}(x,y-1)$.}
Continue factoring in this manner, {reducing $x$ and $y$,} until each block {in the factorization} is one of {$B_\xi'(3,3)=H', B_\xi'(j,2)=F'_j$, or $B_\xi'(2,k)=G'_k$, for some $j,k \geq 4$}.
The {concatenation {that we arrive at} must have the} following structure.  $F'_4$ can be preceded only by $H'$.  For $j>4$, $F'_j$ can be preceded only by $F'_{j-1}$.  $G'_4$ can be followed only by $H'$. For $k>4$, $G'_k$ can be followed only by $G_{k-1}$.
Note that for all $k,j\geq 4$ the concatenation $F'_jG'_k$ does occur in some $B_{\xi'}(x,y)$, for example, in  $B_{\xi'}(k+1,j+1)$.

None of the blocks $H'$, $F'_j$, $G'_k$ begin with $a^2$; for all $l>2$, {$k \geq 4$},  $G'_k$ and $H'$ do not contain $a^l$. So if $l>6$ and $ba^lb$ appears in $B_{\xi'}(x,y)$, then the majority of the $a^l$ portion of $ba^lb$ must appear within one of the factors $F'_j$ of $B_{\xi'}(x,y)$ for some $j\geq 4$.   We can then see that $ba^lb$ appears in one of only four possible ways:
 \begin{enumerate}
 \item across a concatenation $F'_{l+1}G'_k$ for some $k\geq 4$,
 \item across a concatenation $F'_{l+1}H$,
 \item across a concatenation $F'_{l}F'_{l+1}$,
 \item entirely within a block $F'_j$ for some $j>l$.
 \end{enumerate}
 In the first case, the appearance of $ba^lb$ is within the block $ba^lba^lb^{k-1}a$ for some $k\geq 4$, while in the second case we see that the appearance of $ba^lb$ is within the block $ba^lba^lb^2ab$.  In the third case the appearance of $ba^lb$ is within the block $ba^{l-1}ba^lbab$.  In the final case if $j>l+1$ the appearance of $ba^lb$ is within the block $ba^{l-1}ba^{l}ba^{l+1}b$.  However, if $j=l+1$, $F'_{l+1}$ is not long enough to gather enough information about how $ba^lb$ is extended to the right and we must examine which blocks can follow $F'_{l+1}$.  These are $G'_k, H'$, and $F'_{l+2}$.  These give rise to the longer words $ba^lba^lb^{k-1}a$ for some $k\geq 4$, $ba^lba^lb^2a$, and $ba^{l-1}ba^lba^{l+1}b$, respectively.    Therefore, $ba^lb$ appears in any sequence in $\Sigma_{\xi'}$ only as a subblock of one of these four longer words.  Most significantly, we see that any extension of $ba^lb$ in $\mathscr{L}(\Sigma_{\xi})$ will not be in $\mathscr{L}(\Sigma_{\xi'})$.

 {We use a similar argument to determine ways that, for $l>6$, {$ab^la$} may occur in any $B_{\xi}(x,y)$ with $x,y\geq 2$. We find that
 {$ab^la$} appears in any sequence in $\Sigma_{\xi}$ only at the end of the longer block $ab^{l-1}ab^lab^la$ or as a subblock of  $ab^{l-1}ab^{l}a$.}

 {Similarly, for $l>6$, {$ab^la$} may occur in $B_{\xi'}(x,y)$ for $x,y\geq 2$ only as a subblock of one of the four longer blocks $ba^{j-1}b^lab^lab^{l-1}$ for some $j\geq 4$, $ba^2b^lab^lab^{l-1}a$, $babab^lab^{l-1}a$, or $ab^{l+1}ab^lab^{l-1}a$. Most significantly, we see that any extension of $ab^la$ in $\mathscr{L}(\Sigma_{\xi})$ will not be in $\mathscr{L}(\Sigma_{\xi'})$.}

{Let $\omega\in \Sigma_{\xi}\cap\Sigma_{\xi'}$ {and suppose} that $\omega$ is not in the orbit of {any} of {the sequences}
\[a^\infty,\hfill b^\infty, a^\infty . ba^\infty, \hfill b^\infty . ab^{\infty}.\] We will show that then a block $ba^lb$ or $ab^la$ with $l>6$ appears in $\omega$, contradicting what was just established above.}

Since $\omega= \dots \omega_{-2}\omega_{-1}.\omega_0\omega_1 \dots$ {is not in one of the four special orbits listed above, there are at least three distinct values of $i$ for which $\omega_{i}\neq \omega_{i+1}$, say for $i_1<i_2<i_3$;}  {in other words, $\omega$ contains at least 3 ``switches" from $a$ to $b$ or vice versa.}
{Let $S=\omega_{i_1}\dots\omega_{i_3+1}$ be such a block  which contains exactly three switches, so that $S$ has the form $ba^mb^na$ or $ab^na^mb$ for some $m,n\geq 1$.}
{Let $D$ = $\omega_{i_1-100}\dots\omega_{i_1}\dots\omega_{i_3+101}$; in} other words, {$D$} extends {in $\omega$} to the left and the right of the three {places $i_1,i_2,i_3$, where switches are known to occur}, by 100 entries. In addition, let {$E$} be the subblock of $\omega$ that extends $D$ by 20 letters to the left and 20 letters to the right.
 Since $E\in \mathscr{L}(\Sigma_{\xi})$, there is a block $B_{\xi}(x,y)$ {} that contains $E$ (with $x,y \geq 2$, since otherwise $B_\xi(x,y)$ contains only two switches). Moreover, {the block $H$ does not appear} in $\omega$, so $E$ does not contain the entirety of $H$ {as a subblock}.
 Therefore, {when $B_\xi(x,y)$ is written according to the above scheme as a concatenation of $H,F_j,G_k$ for certain choices of $j,k\geq 4$}, the portion of $B_{\xi}(x,y)$ containing $D$ does not overlap any part of {an appearance of $H$ in the factorization (meaning that the intervals of indices of $B_\xi(x,y)$ at which any copies of $D$ and $H$ appear are disjoint)}. In other words, $D$ is contained in its entirety within concatenations of the blocks $F_j$ and $G_k$ according to the structure given above.

Consider the { factorization of $B_\xi(x,y)$ as described above, with $j,k \geq 4$.} {The block $S=\omega_{i_1}\dots\omega_{i_3+1}$ must overlap some blocks  $F_j$ and $G_k$ of the factorization, and this can happen only with $j,k \geq 8$.}  This is because $F_7$ {can appear in any sequence in $\Sigma_\xi$ only as part of the string $F_7F_6F_5F_4H$;} but we know that $D$ does not overlap any portion of $H$, and the length of $F_7F_6F_5F_4$ is 100.
{Similarly, {$S$} cannot overlap $F_6,F_5,$ or $F_4$. And if
  $S$ overlaps a block $G_k$ of the factorization, it must happen with $k\geq 8$, since $G_7$ can occur in any sequence in $\Sigma_\xi$ only at the end of a string $HG_4G_5G_6G_7$, which also has length 100; and similarly for $G_6, G_5,$ and $G_4$.}

 As {mentioned} above, $S$ is either of the form $ba^mb^na$ or $ab^na^mb$. If either $m$ or $n$ is greater than 6, then $ba^lb$ or $ab^la$, with $l>6$, appears in {$S$}, a subblock of {$D$}, {a subblock of $\omega$.}

 {Suppose that both $m,n <7$ and $S$ overlaps $F_j$ for some $j\geq 8$ or $S$ overlaps $G_k$ for some $k\geq8$. Then inspection} of the blocks $F_j$ and $G_k$ with $j,k\geq 8$ {shows} that $S$ must necessarily overlap {the suffix of $F_j$ of length 45, or the prefix of length 45 of $G_k$, or both, for some $j,k\geq 8$.}
 However, if $S$ overlaps the last 45 {entries of an appearance of a block $F_j$ with $j\geq 8$, then the block $ba^7b$ appears in the block $\omega_{i_1 - 45}\dots \omega_{i_1}$. Similarly, if $S$ overlaps the first 45 entries of an appearance of a block $G_k$ with $k\geq 8$, then the block $ab^7a$ appears in the block $\omega_{i_3+1}\dots \omega_{i_3 + 45}$.}
 Thus at least one of {the blocks $ba^7b$ or $ab^7a$  appears in $\omega$.
 We conclude that
 it is impossible that any $\omega\in \Sigma_{\xi}\cap\Sigma_{\xi'}$ is not in the orbit of one of {the sequences}
$a^\infty, b^\infty, a^\infty . ba^\infty,
 b^\infty . ab^{\infty}$.}
\end{proof}

\section{Remarks, Examples, Questions}

\begin{question}
Is coding orbits by the {\em first} edge also essentially faithful for every ordering of the Pascal diagram? In other words, under the hypotheses of Theorem \ref{thm:coding}, is the coding of $\gamma$ by initial segments of length $1$ not the same as the coding of $\gamma'$
 {by initial segments of length $1$}?
\end{question}

\begin{ex}\label{ex:uncountablymanymax}
It is possible to assign an ordering to the Pascal graph in such a way that there are uncountably many minimal paths. It suffices to find a binary tree embedded in the graph, for then we can define the ordering in such a way that all edges in the tree are minimal, the other edges to vertices in the tree are maximal, and the orderings at remaining vertices are assigned arbitrarily. Let the tree include the vertex $(0,0)$ at level 0 and nonintersecting paths to vertices $(3,0)$ and $(1,2)$ at level 3.

Continue inductively as follows. Suppose that $n \geq 2$ and the tree from level 0 to level $2^n-1$ contains every other vertex at level $2^n-1$ and no other vertices at that level or at higher levels. (If these vertices are labeled $v_0=(2^n-1,0), v_1=(2^n-2,1), \dots, v_{2^n-1}=(0,2^n-1)$, we are including $v_0, v_2, \dots, v_{2^n-2}$.) Extend the tree to contain $2^n$ nonintersecting paths from these vertices to every fourth vertex at level $2^{n+1}-2$. Then from each of these vertices at level $2^{n+1}-2$ extend two paths to form nonintersecting paths to every other vertex at level $2^{n+1}$.

We can simultaneously embed another dyadic tree which shares no edges with the first tree---if necessary, spread out the construction of the first tree even more. Then declaring that all edges in the second tree are maximal will produce an ordering for which there are uncountably many maximal paths as well as uncountably many minimal paths.
\end{ex}

\begin{question}
We do not know whether it is possible for an ordering to have countably many maximal paths and uncountably many minimal paths, or vice versa. (But no ordering on the Pascal graph can have only finitely many maximal or minimal paths, since from each vertex there is a unique path up to the root consisting entirely of maximal (or entirely of minimal) edges.)
\end{question}

\begin{question}We are interested in asymptotic estimates of the complexity function $p_\xi(n)$ which counts the number of $n$-blocks seen in the 1-codings of all sequences in the adic system $(X,T_\xi)$ on the Pascal graph with a fixed ordering $\xi$. It is known \cite{MelaPetersen2005} that for the ordinary Pascal system, determined by $\xi \equiv 0$, $p_0(n) \sim n^3/6$. Can some orderings produce asymptotically larger or smaller complexity functions? If orderings are chosen at random, for example by independent choices at each vertex of $0$ with some probability $q \in (0,1)$ and $1$ with probability $1-q$, what is the expected asymptotic complexity of the set of resulting codings?\end{question}

\end{document}